\documentclass{article}

\usepackage{lineno,hyperref}

\usepackage{amsthm}

\usepackage{fullpage}
\usepackage{amsmath}
\usepackage{amsfonts}
\usepackage{amssymb}
\usepackage{graphicx}
\usepackage{color}
\usepackage{epsf}
\usepackage{float}
 \usepackage{ulem}
\usepackage{breqn}
\usepackage{bigints}
\usepackage[english]{babel}
\usepackage[nottoc]{tocbibind}

\theoremstyle{plain}
\newtheorem{theorem}{Theorem}
\newtheorem{proposition}[theorem]{Proposition}

\newcommand{\ba}{\begin{array}}
\newcommand{\ea}{\end{array}}
\newcommand{\beq}{\begin{equation}}
\newcommand{\eeq}{\end{equation}}
\newcommand{\bea}{\begin{eqnarray}}
\newcommand{\eea}{\end{eqnarray}}
\newcommand{\bean}{\begin{eqnarray*}}
\newcommand{\eean}{\end{eqnarray*}}
\newcommand{\bal}{\begin{align}}
\newcommand{\eal}{\end{align}}
\newcommand{\bit}{\begin{itemize}}
\newcommand{\eit}{\end{itemize}}
\newcommand{\benum}{\begin{enumerate}}
\newcommand{\eenum}{\end{enumerate}}
\newcommand{\bdm}{\begin{displaymath}}
\newcommand{\edm}{\end{displaymath}}

\newcommand{\bssz}{\begin{scriptsize}}
\newcommand{\essz}{\end{scriptsize}}
\newcommand{\bfnsz}{\begin{footnotesize}}
\newcommand{\efnsz}{\end{footnotesize}}

\newcommand{\blue}[1]{\textcolor{black}{#1}} 

\newcommand{\half}{\frac{1}{2}}

\newcommand{\tUjpk}{\vect U_{j+\half,k}}

\newcommand{\tUjkp}{\vect U_{j,k+\half}}
\newcommand{\tqjmk}{q_{j-\half,k}}
\newcommand{\tqjpk}{q_{j+\half,k}}
\newcommand{\tqjkm}{q_{j,k-\half}}
\newcommand{\tqjkp}{q_{j,k+\half}}

\newcommand{\detj}{\vert J\vert}
\newcommand{\alfas}{\alpha\left(s\right)}
\newcommand{\alfaps}{\alpha'\left(s\right)}
\newcommand{\alfapps}{\alpha''\left(s\right)}
\newcommand{\cosa}{\cos\left(\alfas\right)}
\newcommand{\sina}{\sin\left(\alfas\right)}
\newcommand{\cost}{\cos\left(\theta\right)}
\newcommand{\sint}{\sin\left(\theta\right)}
\newcommand{\partt}{\partial_{t}}
\newcommand{\parts}{\partial_{s}}
\newcommand{\partr}{\partial_{r}}
\newcommand{\parth}{\partial_{\theta}}
\newcommand{\partx}{\partial_{x}}
\newcommand{\party}{\partial_{y}}
\newcommand{\partz}{\partial_{z}}
\newcommand{\partA}{\partial_{1}}

\newcommand{\partii}{\partial_{2}}
\newcommand{\partiii}{\partial_{3}}
\newcommand{\dermat}{\frac{D}{Dt}}
\newcommand{\lap}{\Delta}
\newcommand{\grad}{\nabla}
\newcommand{\vect}[1]{{\bf {#1} }}
\graphicspath{{pictures/}}

\newcommand{\tdermat}{\tilde{\dermat}}
\newcommand{\trho}{\tilde{\rho}}
\newcommand{\tVs}{\tilde{V_s}}
\newcommand{\tVr}{\tilde{V_r}}
\newcommand{\tVt}{\tilde{V_\theta}}
\newcommand{\tir}{\tilde{r}}
\newcommand{\tiP}{\tilde{P}}
\newcommand{\tdetj}{\tilde{\detj}}
\newcommand{\tparts}{\tilde{\parts}}
\newcommand{\tpartr}{\tilde{\partr}}
\newcommand{\tparth}{\tilde{\parth}}
\newcommand{\tvector}{\tilde{\vect V}_c}
\newcommand{\tgradc}{\tilde{\grad}_c}
\newcommand{\phat}{\widehat{p}}
\newcommand{\pbar}{\overline{p}}
\newcommand{\detjr}{\frac{\detj}{r}}
\newcommand{\rdetj}{\frac{r}{\detj}}
\newcommand{\vs}{V_s}
\newcommand{\vr}{V_r}
\newcommand{\vth}{V_\theta}
\newcommand{\tdetjr}{\frac{\tdetj}{\tilde{r}}}
\newcommand{\gs}{\gamma_s}
\newcommand{\gt}{\gamma_\theta}

\title{A new two-dimensional blood flow model with arbitrary cross sections
\thanks{
C. Rosales-Alcantar was supported by Conacyt National Grant 701892. G. Hernandez-Duenas was supported in part by grants UNAM-DGAPA-PAPIIT IN113019 \& Conacyt A1-S-17634. 
}
}

\date{}

\author{Cesar Alberto Rosales-Alcantar \& Gerardo Hernandez-Duenas \\
Institute of Mathematics, National University of Mexico \\ 
 Blvd. Juriquilla 3001, Queretaro, Mexico.
 \thanks{Email: cesar@im.unam.mx}}

\begin{document}

\maketitle
\begin{abstract}
A new two-dimensional model for blood flows in arteries with arbitrary cross sections is derived. The model consists of a system of balance laws for conservation of mass and balance of momentum in the axial and angular directions. The equations are derived by applying asymptotic analysis to the incompressible Navier-Stokes equations in narrow, large vessels and integrating in the radial direction in each cross section. The main properties of the system are discussed and a positivity-preserving well-balanced central-upwind scheme is presented. The merits of the scheme will be tested in a variety of scenarios. In particular, numerical results of simulations using an idealized aorta model are shown. We analyze the time evolution of the blood flow under different initial conditions such as perturbations to steady states consisting of a bulging in the vessel's wall. We consider different situations given by distinct variations in the vessel's elasticity.
\end{abstract}
%
\pagestyle{plain}

\section{Introduction}
\label{intro}
The impact of cardiovascular diseases in our lives has motivated the development of different models for blood flows. In \cite{formaggia2000advances}, a review of recent contributions towards the modeling of vascular flows is provided. A review of contributions regarding the mathematical modeling of the cardiovascular system is presented in \cite{quarteroni_manzoni_vergara_2017}. In particular, the challenges in the mathematical modeling both for the arterial circulation and the hearth function are discussed. See also the notes in \cite{quarteroni2004mathematical} for more on mathematical modeling and numerical simulation of the cardiovascular system. In \cite{arthurs2021crimson}, an open-source software framework for cardiovascular integrated modeling and simulation (CRIMSON) is described, which a tool to perform three-dimensional and reduced-order computational haemodynamics studies for real world problems.  \\



Three dimensional (3D) models provide very detailed information of the fluid's evolution. In \cite{ku1985pulsatile}, fluid velocities were measured by laser Doppler velocimetry under conditions of pulsatile flow and such measurements were compared to those given by steady flow conditions. In \cite{taylor2009patient}, anatomic and physiologic models are obtained with the aid of 3D imaging techniques for patient-specific modeling. However, 3D simulations are computationally expensive and often not a practical tool for a timely evaluation before a surgical treatment.  This has motivated the development of 1D models which reasonably well describe the propagation of pressure waves in arteries~\cite{Formaggia2003}. A 1D hyperbolic model for compliant axi-symmetric idealized vessels is derived in \cite{vcanic2003mathematical} and the properties of the model are discussed. The analysis of the blood flow after an endovascular repair is studied in \cite{Canic2002}, in which case the PDE based model has discontinuous coefficients. Furthermore, effects of viscous dissipation, viscosity of the fluid and other two dimensional effects were incorporated in a `one-and-a-half dimensional' model in \cite{vcanic2006blood}, where it is not necessary to prescribe an axial velocity profile \textit{a priori}. In \cite{canic2017dimension}, a coupled model that describes the interaction between a shell and a mesh-like structure is derived. The structure consists of 3D mesh-like elastic objects and the model embodies a 2D shell model and a 1D network model. Well-balanced high-order numerical schemes for one-dimensional (1D) blood flow models are constructed in \cite{MULLER201353} using the Generalized Hydrostatic Reconstruction technique.  Arbitrary Accuracy Derivative (ADER) finite volume methods for hyperbolic balance laws with stiff terms are extended to solve one-dimensional blood flows for viscoelastic vessels in \cite{Montecinos2014}, and such technique is applied to analyze the treatment of viscoelastic effects at junctions in \cite{Lucas2016}. The effects of variations of the mechanical properties of arteries due to diseases such as stenosis or aneurysms have also being studied using 1D models \cite{Willemet2015,Vosse2011,Ku1997}. It has also been noted in \cite{huberts2012experimental} that 1D models are able to describe the fluid's evolution after arteriovenous fistula (AVF) surgeries in 6 out of 10 patients and selected the same AVF location as an experienced surgeon in 9 out of 10 patients. Although 1D models have shown to be a reasonably good approximation in many situations, there exist limitations due in part to simplifications such as axial symmetry. The cross section is assumed to be a circle, which impacts the results. \\

The contributions listed above consider the two extreme cases where the models are either 3D or 1D, and some interactions between them. In this work, we derive an intermediate two-dimensional (2D) model where any shape of the cross section can be considered while maintaining a still much lower computational cost compared to 3D simulations. The 2D models are a good balance that provides more realistic results compared to its 1D counterparts and it has the advantage of a low computational cost when compared to the 3D models. The derivation and implementation of a model in two dimensions is one of our main contributions. The model is derived using asymptotic analysis that follows certain physical considerations. The velocity and vessel's radius here depend on the cross section's angle and axial position, while the 1D counterpart considers a uniform radius that varies only in the axial direction. Our model can handle perturbations and variations in the wall's elasticity affecting any specific area of the vessel while 1D simulations can only consider perturbations affecting entire cross sections. This is relevant in simulations of diseases such as aneurysms and stenosis that involve vessels with walls that have damaged areas, not necessarily entire cross sections. Furthermore, we present the properties of the model, construct a well-balanced central-upwind scheme and include numerical tests that show its merits.  \\

Our work is organized as follows. Section \ref{sec:Derivation} provides the derivation of the system of partial differential equations that describes our model, leaving the details of the asymptotic expansion to \ref{appendix:Derivation}. The model is derived by applying asymptotic analysis to the Navier-Stokes equations written in cylindrical coordinates and integrating in the radial direction in each cross section. Section \ref{sec:Properties} describes the quasilinear properties of the model, where we show that the system is conditionally hyperbolic. A closed form of the transmural pressure term is presented. Section \ref{sec:NumScheme} focuses on the description of the positivity-preserving well balanced second order central-upwind scheme. Finally, in Section \ref{sec:NumTests} we test our model by considering a variety of numerical experiments.

\section{Derivation of the model}
\label{sec:Derivation}
 
\subsection{The geometry of the vessel}

The model presented here is derived by computing the radial average in each axial position and angle in the vessel's cross section. Although it can be easily generalized, let us assume for simplicity that the vessel is aligned in the $x-z$ plane and that it extends along a curve that passes through each cross section. The parametrization of such curve is assumed to be known and its location is then represented by its arclength's position $s$ and coordinates $(x_o(s),y_o = 0, z_o(s))$. Each cross section is identified with the arclength position $s$ of the curve intersecting it. Any position in each cross section is located with the angle $\theta$ formed between the displacement from the intersection and a reference vector. The variables and parameters are functions of the axial position $s$ and angle $\theta$. As a result, the model allows for arbitrary cross sections and variations in each angle $\theta$, resulting in a 2D model. \\

 The above parametrization is done such that $s=0 $ ($s=s_L$) corresponds to the left (right) end of the vessel. For each $s$, the cross-section denoted by $C(s)$ is contained in a plane passing through $(x_o(s),y_o = 0, z_o(s))$ and perpendicular to the unit tangent vector $\vect T(s) = (\cosa, 0, \sina)$. Here $\alfas$ is the angle of the curve with respect to the horizontal axis $x$. Furthermore, for a point $(x,y,z)$ in the cross-section $C(s)$, let $\theta$ be the angle between the normal vector $(-\sina,0,\cosa)$ and the displacement $(x,y,z)-(x_o(s),0, z_o(s))$. This gives the following change of variables:
\begin{align}\label{ncs}
x\left(r,s,\theta\right) & = \:r\sina\sint + x_o\left(s\right), \nonumber \\
y\left(r,s,\theta\right) & = \: r\cost, \\
z\left(r,s,\theta\right) & = \: r\cosa\sint + z_o\left(s\right), \nonumber
\end{align}
where $r$ is the norm of the displacement. The corresponding Jacobian is given by
\[
\detj = r\left(1-r\sint\alfaps\right).
\]
A sufficient condition for the change of variables to be valid is that the radius $r$ for any point in the cross section does not exceed the radius of curvature of the parametrization. That is,
\[
r \le \mathfrak{R}(s) = \frac{1}{\kappa(s)},
\]
where $\kappa(s) = |\alpha'(s)|$ is the vessel's curvature and $\mathfrak{R}$ is the radius of curvature. \\

\begin{figure}[h!]
\begin{center}
{\includegraphics[width=0.7\textwidth]{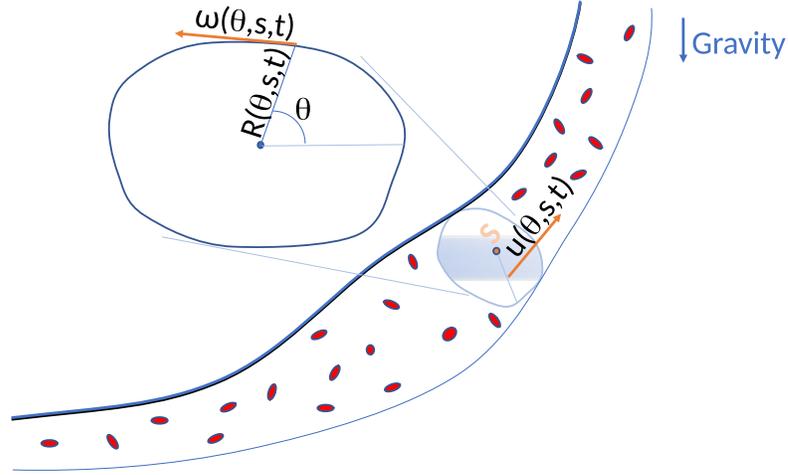}}
\end{center}
\caption{\label{fig:schematic} Schematic of model for blood flows passing through a compliant vessel.
}
\end{figure} 
Figure \ref{fig:schematic} shows the schematic of the model to be derived below. The vessel's radius may vary as a function of angle, axial position and time, $R=R(t\:;\:\theta,s)$. As a result, the vessel may have any cross-sectional shape. On the contrary, one-dimensional models assume a uniform radius, independent of $\theta$, restricting the cross section to be axi-symmetric.

\subsection{The averaged leading order equations}
\label{sec:AveragedEqns}

The derivation of the model requires rewriting the equations in cylindrical coordinates and carry out an asymptotic analysis to determine the leading order contribution under the following assumptions. Let $V_{s,o}$, $V_{r,o}$, and $V_{\theta,o}$ be the characteristic radial, axial and angular velocities. Let also $\lambda$ and $R_o$ be the characteristic axial and radial lengthscales. The small parameter in this expansion is the ratio between radial and axial lengthscales
\begin{equation}
\label{eq:Epsilon}
\epsilon : = \frac{R_o}{\lambda} = \frac{V_{r,o}}{V_{s,o}}.
\end{equation}
Typical values in the aorta between the renal and the iliac arteries gives $\epsilon \approx 10^{-2}$ \cite{vcanic2003mathematical}. Furthermore, we assume that the gravity $g$, the scales of pressure ($[P]$), radius ($R_o$), time ($T$), axial  and linear angular velocities ($V_{s,o},V_{\theta,o}$)  satisfy
\begin{equation}
\label{eq:Assumption1}
\frac{[P]}{\rho_o V_{s,o}^2} = O(1), \; \frac{V_{s,o}}{R_o V_{\theta,o}} = O(1), \; \frac{gT}{V_{s,o}}= O(1).
\end{equation}
Under these assumptions, the acceleration of gravity is comparable to the characteristic acceleration of the system in the axial direction. This is reasonable for typical velocities of order $O(1 \text{ ms}^{-1})$ and a timescale of $T=0.1 \text{ s}$. On the other hand, the change of variables is valid provided that $r |\alpha'(s)| < 1$. As a stronger assumption, we assume that $r \alpha'(s)$ is small, which implies that the artery's radius of curvature is large compared to its cross-sectional radius. On the other hand, an approximate value for blood viscosity in arteries can be taken as a constant $\nu = 4 \text{ cP} = 4 \times 10^{-2} \text{ g}\text{ (s cm)}^{-1}$ \cite{Ku1997}. Using $\rho_o = 1050 \text{ kg m}^{-3}$, and $R_o = 0.82 \text { cm}$, it gives us $\frac{\nu T}{\rho_o R_o^2} = 0.56 \times 10^{-2}$. Based on this estimation, we assume
\begin{equation}
\label{eq:Assumption2}
R_o \left| \alpha'(s) \right| = O\left( \epsilon \right),\:  \text{ and }\:\: \frac{\nu T}{\rho_o R_o^2} = O\left( \epsilon \right).
\end{equation}

After removing terms of order $O(\epsilon^2)$ in the non-dimensionalized equations, we obtain
\begin{equation}
\label{eq:ReducedEqns}
\begin{array}{lcl}
\detj\dermat\left(\rho\right) & = & 0, \\ \\
\detj\dermat\left(\rho\left[\detjr\right]^2\vs\right) & = & r\left[\detjr\right]^2\parts\left(\detjr\right)\rho\vs^2 - \detj\parts\left(p\right)-\detj \sina \rho g \\
& & + \nu \partr\left(\detj\partr\left(\left[\detjr\right]^2\vs\right)\right), \\ \\
\detj\dermat\left(\rho r^2\vth\right) & = & r\left[\detjr\right]^2\parth\left(\detjr\right)\rho\vs^2 -\detj \parth\left(p\right) \\
& & + \nu\partr\left(\detj\left[\partr\left(r^2\vth\right) - 2r\vth\right]  \right), \\ \\
\grad_c \cdot\left(\detj \vect V_c\right) & = & 0.
\end{array}
\end{equation}
The radial averaging is applied to the limiting equations. The details of the reduction of the system are left to \ref{appendix:Derivation} .

\subsection{The main system}

Let $R\left(s,\theta,t\right)$ denote the vessel's cross-sectional radius at each position $s$, angle $\theta$ and time $t$. Let us define
\begin{equation}
A\left(R,s,\theta,t\right) = \int_0^R \detj dr, \:\:\:\:\:\: u\left(s,\theta,t\right) = \frac{1}{A}\int_0^R \vs \detj dr,
\end{equation}
\begin{equation}
\omega\left(s,\theta,t\right) = \frac{1}{A}\int_o^R \vth \detj dr, \:\:\:\:\:\: \text{and } L\left(R,s,\theta,t\right) = \frac{1} A \int_o^R r^2 V_\theta \detj dr,
\end{equation}
where 
\[
A = A(R,s,\theta) = \frac{R^2}{2}-\frac{R^3}{3} \sin(\theta) \alpha'(s)
\] 
is the radially integrated Jacobian. On the other hand, $u, \omega$ and $L$ are the radially-averaged axial velocity, angular velocity and angular momentum, respectively. \\
\\
The model is derived after an integration in the radial direction, assuming a streamline
\[
\left[\vr\right]_{r=R} = \partt\left(R\right) + \left[\vs\right]_{r=R}\parts\left(R\right) + \left[\vth\right]_{r=R}\parth\left(R\right)
\]
at the artery's wall and a slowly varying density, which is approximated by a constant value. The main system is
\begin{equation}
\label{eq:MainSystemNotConservation}
\begin{array}{llll}
\partt\left(A\right) + \parts\left( A \: u \right) + \parth\left( A \: \omega \right) & = & 0, \\ \\
\partt\left(\psi_{s,o} A u \right) + \parts\left( \psi_{s,1} A u^2 \right) + \parth\left( \psi_{s,2} A u \omega \right) & = & -A \parts\left(\frac{p}{\rho}\right) - g A \sin(\alpha(s)) \\
&& + \int_o^R r \left[ \frac{|J|}{r} \right]^2  \partial_s \left( \frac{|J|}{r} \right)V_s^2 dr \\
&& + \frac{\nu}{\rho} \left(|J| \partial_r \left( \left[ \frac{|J|}{r}\right]^2 V_s \right)\right) \left|_o^R \right., \\ \\
\partt\left(A L \right) + \parts\left( \psi_{\theta,1} A L u \right) + \parth\left( \psi_{\theta,2} A L \omega \right) & = & -A\parth\left(\frac{p}{\rho}\right) \\
&& + \int_o^R r \left[ \frac{|J|}{r} \right]^2 \partial_\theta \left( \frac{|J|}{r} \right)V_s^2 dr \\
&& + \frac{\nu}{\rho} \left(|J| \left[ \partial_r ( r^2 V_\theta ) - 2 r V_\theta \right]\right)  \left|_o^R \right. , \\
\end{array}
\end{equation}
where $\psi_{s,o},\psi_{s,1},\psi_{s,2},\psi_{\theta,1}$ and $\psi_{\theta,2}$ are Coriolis terms satisfying
\begin{equation}
\begin{array}{lcllcllcl}
\psi_{s,o} A u & = & \int_o^R |J| \left[ \frac{|J|}{r} \right]^2 V_s dr , & &&  \\ \\
\psi_{s,1} A u^2 & = & \int_o^R |J| \left[ \frac{|J|}{r} \right]^2 V_s^2 dr , & \psi_{s,2} A u \omega & = & \int_o^R |J| \left[ \frac{|J|}{r} \right]^2 V_s V_\theta dr, \\ \\
\psi_{\theta,1} A L u & = & \int_o^R |J| r^2 V_\theta V_s dr, & \psi_{\theta,2} A L \omega & = & \int_o^R |J| r^2 V_\theta^2 dr.
\end{array}
\end{equation}

The angular velocity and angular momentum are related by
\[
L = A_\theta \omega,\:\:\:\: A_\theta = \frac{\int_o^R r^2 \vth \detj dr}{\int_o^R \vth \detj dr}.
\]

The integral of the Jacobian in the radial direction $A$ is one of our conserved variables. This variable $A$ satisfies that $\int_{s_0}^{s_1} \int_0^{2\pi} A d\theta ds$ is the volume in the corresponding artery's region. In the case of a circular cross section, the integral with respect to $\theta$ gives us the cross-sectional area. The balance of axial and angular momenta determine the other two conserved quantities, given by
\[
Q_1 = \psi_{s,o}  A u, Q_2 = A L,
\]
where $\psi_{s,o}$ takes into account the effect of curvature in the artery, and $\psi_{s,o}=1$ for horizontal vessels. One still needs to determine a profile for the axial and angular velocities $V_s$ and $V_\theta$ as functions of $r$ to close the system. For a given profile, the non-dimensional Coriolis terms $\psi_{s,o},\psi_{s,1},\psi_{s,2},\psi_{\theta,1}$ and $\psi_{\theta,2}$ are all explicit functions of $A, s$ and $\theta$. In fact, those parameters are explicit functions of $R\sin (\theta) \alpha'(s)$ for the profiles considered in this paper, and are therefore constant parameters for vessels with zero curvature. See section \ref{sec:SpecificProfiles} for more details. The transmural pressure, which is the pressure difference between the two sides of the artery's wall is denoted by $p$. The elasticity properties of the vessel are determined by a relationship between $p$, the area $A$ and possibly the variables $s$ and $\theta$ due to non-uniform properties in the artery (explicit dependance on parameters). That is, we assume that the transmural pressure is an explicit function 
\[
p = p(A,s,\theta).
\]
Furthermore, we assume that the transmural pressure vanishes at a given state at rest $A_o = A_o(s,\theta) = \frac{R_o^2}{2}-\frac{R_o^3}{3} \sin(\theta) \alpha'(s)$ with $R_o = R_o(s,\theta)$ such that
\[
p(A_o(s,\theta), s,\theta ) = 0.
\]
Some properties such as the hyperbolicity of the system are shown independently of the profiles, and we take equation \eqref{eq:MainSystemNotConservation} as the most general form of the system. This closes the system since everything is given in terms of the conserved quantities $A,Q_1$ and $Q_2$, and $s,\theta$ due to the presence of varying parameters. \\

\noindent
{\bf Conservation form}\\

The model can be written in conservation form. For that end, we need to introduce the following notation. The transmural pressure and other parameters such as the Coriolis terms are explicit functions $q=q(A,s,\theta)$ of $A,s$ and $\theta$. On the other hand, model parameters such as $R_o$ and $A_o$ depend explicitly on $(s,\theta)$. We will denote by $\partial_1, \partial_2, \partial_3$ the derivatives with respect to $A,s$ and $\theta$, keeping the other terms fixed. Thus
\[
\partial_s (p(A(s,\theta,t),s,\theta)) = \partial_1 p \: \partial_s A + \partial_2 p, \; \partial_\theta (p(A(s,\theta,t),s,\theta)) = \partial_1 p \: \partial_\theta A + \partial_3 p. 
\]
We distinguish them from $\partial_s$ and $\partial_\theta$, which take into account the variations of the conserved variables with respect to $s$ and $\theta$ over time.\\

In order to get the conservation form, we define the splitting of the transmural pressure as
\[
p = \phat + \pbar,
\]
where
\[
\phat\left(A,s,\theta\right) = \frac{1}{A}  \int _{A_o}^A \mathcal A \partial_1 \left(p\left(\mathcal A,s,\theta \right)\right) d \mathcal A,
\]
which satisfy
\[
\begin{array}{ccc}
A\parts\left(p\right) & = & \parts\left(A\phat\right) + A \partial_2 \pbar, \text{ and }\\ \\
A\parth\left(p\right) & = & \parth\left(A\phat\right) + A \partial_3 \pbar. 
\end{array}
\]

The 2D model for blood flows in arteries with arbitrary cross sections is written as a hyperbolic system of balance as 
\begin{equation}\label{eq:ConservationComplete}
\partt \vect U + \parts \: \vect F\left(\vect U\right) + \parth \: \vect G\left(\vect U\right)  = S\left( \vect U\right),
\end{equation}
where
\begin{equation}
\label{eq:ConservedQFlux}
\begin{array}{l}
\vect U = 
\left(
\begin{array}{c} 
A \\
\psi_{s,o} A u \\
A L 
\end{array}
\right),\:\:
\vect F \left(\vect U\right) = 
\left(
\begin{array}{c} 
A u \\
\psi_{s,1} A u^2 + \frac{1}{\rho} A\widehat{p} \\
\psi_{\theta,1} A u L
\end{array}
\right),\:\: \text{ and } \; \\
\vect G\left(\vect U\right) = 
\left(
\begin{array}{c} 
A \omega \\
\psi_{s,2} A u \omega \\
\psi_{\theta,2} A L \omega + \frac{1}{\rho}A \widehat{p} 
\end{array}
\right)
\end{array}
\end{equation}
are the vectors of conserved variables and the fluxes in the axial and angular directions, respectively. The vector of source terms is 
\begin{equation}
\label{eq:SourceGeneral}
\vect S\left(\vect U\right) = 
\left(
\begin{array}{c}
0 \\
\begin{array}{rl}
-\frac{A}{\rho} \partial_2 \pbar - A \sina g & +  \int_o^R \detjr \partial_s \left( \frac{|J|}{r} \right)V_s^2 \detj dr \\
&+ \frac{\nu}{\rho}\left[\detj\partr\left(\left[\detjr\right]^2 V_s\right)\right]_{r=R} 
\end{array}  \\
\begin{array}{rl}
-\frac{A}{\rho} \partial_3 \pbar + \int_o^R \detjr \parth & \left( \frac{|J|}{r} \right)V_s^2 \detj dr \\
& + \frac{\nu}{\rho}\left[\detj\partr\left(r^2 V_\theta\right) - 2r\detj V_\theta\right]_{r=R}
\end{array} \\
\end{array}
\right).
\end{equation}
We note that none of the source terms are non-conservative products. The source terms only involve derivatives of the model parameters with respect to the explicit dependance on $(s,\theta)$ and no derivates of the solution itself are present. This prevents both theoretical and numerical complications when shockwaves arise. Below we show explicit expressions of the source and Coriolis terms for a particular choice of profiles for the transmural pressure, axial and angular velocities. As we will see, the expressions are very simple in the case of horizontal arteries ($\alfaps=0$).

\section{Properties of the model}
\label{sec:Properties}

\subsection{Hiperbolicity of the model}


The conservation form in equation \eqref{eq:ConservationComplete} is crucial because it allows us to formulate the Rankine-Hugoniot conditions for weak solutions in the presence of shockwaves. We can apply the theory of weak solutions provided the model is hyperbolic. The hyperbolic properties of system \eqref{eq:ConservationComplete} can be studied through its quasilinear formulation, which is given by
\begin{equation}
\label{eq:QuasiLinearForm}
\partt \vect U + M_s\left(\vect U\right)\parts \vect U + M_\theta\left(\vect U\right) \parth\vect U = \tilde{\vect S}\left(\vect U\right),
\end{equation}
where the coefficient matrices are
\begin{equation}
\label{eq:CoeffMat1}
M_s\left(\vect U\right)=
\begin{pmatrix}
-\frac{A}{\psi_{s,o}}\partA\left(\psi_{s,o}\right)u & \frac{1}{\psi_{s,o}} & 0 \\
\frac{1}{\rho}\partA\left(A\phat\right) -\left(\frac{\psi_{s,1}}{\psi_{s,o}}\right)^2\partA\left(\frac{\psi_{s,o}^2 A}{\psi_{s,1}}\right)u^2 & 2\frac{\psi_{s,1}}{\psi_{s,o}}u  & 0 \\
-\frac{\psi_{\theta,1}^2}{\psi_{s,o}}\partA\left(\frac{\psi_{s,o} A}{\psi_{\theta,1}}\right)uL & \frac{\psi_{\theta,1}}{\psi_{s,o}}L & \psi_{\theta,1} u
\end{pmatrix},
\end{equation}
and
\begin{equation}
\label{eq:CoeffMat2}
M_\theta\left(\vect U\right)=
\begin{pmatrix}
-\frac{A}{A_\theta}\partA\left(A_\theta\right) \omega & 0 & \frac{1}{A_\theta} \\
-\frac{\psi_{s,2}^2}{\psi_{s,o} A_\theta}\partA\left(\frac{\psi_{s,o} A A_\theta}{\psi_{s,2}}\right)u\omega & \frac{\psi_{s,2}}{\psi_{s,o}}\omega & \frac{\psi_{s,2}}{A_\theta} u \\
\frac{1}{\rho}\partA\left(A\phat\right) - \psi_{\theta,2}^2\partA\left(\frac{A A_\theta}{\psi_{\theta,2}}\right)\omega^2 & 0 & 2\psi_{\theta,2}\omega
\end{pmatrix}.
\end{equation}

The vector of source terms of the quasi-linear formulation is
\begin{equation}
\tilde{S}\left(\vect U\right) = S\left(\vect U\right) +
\begin{pmatrix}
\frac{A}{\psi_{s,o}}\partii\left(\psi_{s,o}\right)u + \frac{A}{A_\theta}\partiii\left(A_\theta\right) \omega \\
-\frac{1}{\rho}\partii\left(A\phat\right) + \left(\frac{\psi_{s,1}}{\psi_{s,o}}\right)^2\partii\left(\frac{\psi_{s,o}^2 A}{\psi_{s,1}}\right)u^2 + \frac{\psi_{s,2}^2}{\psi_{s,o} A_\theta}\partiii\left(\frac{\psi_{s,o} A A_\theta}{\psi_{s,2}}\right)u\omega \\
-\frac{1}{\rho}\partiii\left(A\phat\right) + \frac{\psi_{\theta,1}^2}{\psi_{s,o}}\partii\left(\frac{\psi_{s,o}A}{\psi_{\theta,1}}\right)uL + \psi_{\theta,2}^2\partiii\left(\frac{A A_\theta}{\psi_{\theta,2}}\right)\omega^2 \\
\end{pmatrix}.
\end{equation}

The matrices $M_s$ and $M_\theta$ have two null entries in one column and  their eigenvalues have explicit expressions given by
\begin{equation}\label{eigen_s}
\lambda_o^s = \psi_{\theta,1} u, \; \; \lambda_{\pm}^s = \frac{2\psi_{s,1} - A\partA\left(\psi_{s,o}\right)}{2\psi_{s,o}}u \:\: \pm \:\: \sqrt{\frac{\frac{1}{\rho}\partA\left(A\phat\right)}{\psi_{s,o}} + \Upsilon_1u^2 },
\end{equation}
for $M_s$, and
\begin{equation}\label{eigen_theta}
\lambda_o^\theta = \frac{\psi_{s,2}}{\psi_{s,o}}\omega, \; \; 
\lambda_\pm^\theta = \frac{2\psi_{\theta,2}A_\theta-A\partA\left(A_\theta\right)}{2A_\theta}\omega \pm \sqrt{\frac{\frac{1}{\rho}\partA\left(A\phat\right)}{A_\theta} + \Upsilon_2 \omega^2} 
\end{equation}
for $M_\theta$, respectively. Here, 
\begin{equation}
\begin{array}{lcl}
\Upsilon_1 & = & \frac{1}{\psi_{s,o}^2} \left[\psi_{s,1} - \frac{1}{2}A\partA\left(\psi_{s,o}\right)\right]^2 +  \frac{1}{\psi_{s,o}}\left[A\partA\left(\psi_{s,1}\right) - \psi_{s,1}\right] , \text{ and }\\ \\
\Upsilon_2 & = & \frac{1}{A_\theta^2} \left[\psi_{\theta,2}A_\theta-\frac{1}{2}A\partA\left(A_\theta\right)\right]^2 + \frac{1}{A_\theta} \left[A\partA\left(\psi_{\theta,2}A_\theta\right) - \psi_{\theta,2}A_\theta\right]
\end{array}
\end{equation}
are all non-dimensional quantities.

Below we specify the profiles for the axial and angular velocities. For those specific profiles, we show that $\Upsilon_1= \Upsilon_1(\Gamma)$ and $\Upsilon_2 = \Upsilon_2(\Gamma)$ are explicit rational functions of $\Gamma$, where
\begin{equation}
\Gamma = R \sin(\theta) \alpha'(s).
\end{equation}
Such functions satisfy $\Upsilon_1(\Gamma=0) = \psi_{s,1} (\psi_{s,1}-1)$ and $\Upsilon_2(\Gamma=0) = \left( \psi_{\theta,2}-\frac{1}{2} \right)^2$. The special case $\Gamma=0$ corresponds to a horizontal vessel.   

\begin{proposition}
\label{prop:Hyp}
Let us assume that $p$ is strictly increasing with respect $A$, with positive partial derivative, except possibly at $A=0$ where the vessel collapses. The coefficient matrices $M_s$ and $M_\theta$ given by \eqref{eq:CoeffMat1} and \eqref{eq:CoeffMat2} have real eigenvalues and a complete set of eigenvectors, subject to the conditions
$0< R < \frac{1}{|\alpha'(s)|}$, and $\Upsilon_1, \Upsilon_2 \ge 0$.
\end{proposition}

The condition $0 < R < \frac{1}{|\alpha'(s)|}$ indicates that the artery's radius must not exceed the artery's radius of curvature and it is required for the change of variables to cylindrical coordinates to be valid. This implies that the non-dimensional parameter $\Gamma$ satisfies $|\Gamma| \le 1$.  For specific profiles used in this paper for the axial and angular velocities, $\Upsilon_1$ and $\Upsilon_2$ are in fact non-negative for $ \Gamma \in [-1,1]$.

\begin{proof}
Under the above hypothesis, and the fact that
\[
\partial_1 (A \hat p) = A\partial_1 p \ge 0,
\] 
the expressions inside the square roots in equations \eqref{eigen_s} and \eqref{eigen_theta} are non-negative. As a result, all the eigenvalues are real. The expressions inside the square roots could only vanish if $A=0$ (or equivalently $R=0$). It would also require that $u=0$ or $\omega = 0$ if $\Upsilon_1 = 0$ or $\Upsilon_2 = 0$ respectively. This could happen for certain parameter choices, specially in horizontal vessels. In any case, the condition $R > 0$ is sufficient to guarantee that $\lambda^s_+ \neq \lambda^s_-$ and $\lambda^\theta_+ \neq \lambda^\theta_-$. However, the eigenvalue $\lambda^s_o$ may still have multiplicity 2 if it coincides with $\lambda^s_+$ or $\lambda^s_-$. And the same applies for the matrix $M_\theta$. \\

In the case when the eigenvalues for $M_s$ are different and writing $M_s = (a_{ij})$, the eigenvectors form a basis and are given by
\[
\vect v_o^s = 
\left(
\begin{array}{c}
0 \\
0 \\
1
\end{array}
\right)
,\:\:\:\:\:
\vect v_\pm^s = 
\left(
\begin{array}{c}
a_{12} \\
\lambda_{\pm}^s - a_{11} \\
\frac{\left(a_{11}-\lambda_{\pm}^s\right)a_{32} - a_{12}a_{31}}{a_{33}-\lambda_{\pm}^s}
\end{array}
\right).
\]
If $\lambda_o^s = \lambda_-^s$, the eigenvectors are given by
\[
\vect v_o^s = 
\left(
\begin{array}{c}
0 \\
0 \\
1
\end{array}
\right)
,\:\:\:\:\:
\vect v_-^s = 
\left(
\begin{array}{c}
a_{12} \\
\lambda_-^s-a_{11} \\
0
\end{array}
\right)
,\:\:\:\:\:
\vect v_+^s = 
\left(
\begin{array}{c}
a_{12} \\
\lambda_+^s-a_{11} \\
\frac{\left(a_{11}-\lambda_+^s\right)a_{32} - a_{12}a_{31}}{a_{33}-\lambda_+^s}
\end{array}
\right).
\]
Since $\lambda^s_+ \neq \lambda^s_-$, one can easily check that the eigenvectors form a complete basis because $a_{12} = \frac{1}{\psi_{s,o}}>0$. \\
\\
The case $\lambda_o^s = \lambda_+^s$ and the analysis for $M_\theta$ are analogous.\\

\end{proof}

The hyperbolicity of the model requires that $n_s M_s+n_\theta R_o M_\theta$ has real eigenvalues and a complete set of eigenvectors for all $n_s,n_\theta \in \mathbb R$ such that $n_s^2+n_\theta^2 = 1$. Here $R_o$ is a constant in units of length that appears due to the fact that $s$ and $\theta$ have different units. For the general case, the matrix $n_s M_s+n_\theta R_o M_\theta$ has not a simple form. However, one can easily analyze the special case of a horizontal vessel ($\alpha'(s) = 0$) and $\omega=0$. In such case, $\psi_{s,o},\psi_{s,1},\psi_{s,2},\psi_{\theta,1}$ are all constant, and the characteristic polynomial is
\[
P(\lambda) = -\lambda^3 + c_2 \lambda^2 + c_1 \lambda +c_o,
\]
where
\[
\begin{array}{lcl}
c_2 & = & \left( 2 \psi_{s,1}+\psi_{\theta,1} \right) u n_s, \\
c_1 & = & \frac{1}{\rho}\partial_1 \left( A \hat p \right) \left( n_s^2 +\frac{R_o^2}{A_\theta} n_\theta^2 \right) 
+ \left( -\psi_{s,1} -2 \psi_{s,1}\psi_{\theta,1} \right) \; u^2 \; n_s^2 ,\\
c_o & = & \left[ -\frac{1}{\rho} \partial_1 \left( A \hat p \right) + \psi_{s,1} u^2 \right] \psi_{\theta,1} \; n_s^3 \; u \\
&& + \frac{1}{\rho} \partial_1 \left( A \hat p \right) \frac{R_o^2}{A_\theta} (-2\psi_{s,1}+\psi_{s,2}) \; n_\theta^2 \; u \; n_s.
\end{array}
\]


Using Cardano's approach, we know the characteristic polynomial has three distinct real eigenvalues if
\[
\begin{array}{lcl}
0 & < &  -27 c_o^2 -18 c_2 c_1 c_o +4 c_1^3 -4 c_2^3 c_o + c_2^2 c_1^2 \\
& = & 4 \left( \frac{1}{\rho} \partial_1 (A \hat p) +\psi_{s,1}(\psi_{s,1}-1) u^2 \right) \left[ \frac{1}{\rho} \partial_1 (A \hat p) + (2 \psi_{s,1} \psi_{\theta,1}-\psi_{s,1}-\psi_{\theta,1}^2) u^2 \right]^2 n_s^6\\
&& + 12 \frac{ \left( \frac{1}{\rho} \partial_1  (A \hat p)\right)^3 R_o^2}{A_\theta} n_s^4 n_\theta^2 \\
&& + 4 \frac{ \left( \frac{1}{\rho} \partial_1 (A \hat p) \right)^2 R_o^2 u^2}{A_\theta} \left(   20 \psi_{s,1}^2+\psi_{\theta,1} (9\psi_{s,2}+5\psi_{\theta,1})-\psi_{s,1} (6+9\psi_{s,2}+19 \psi_{\theta,1}) \right) n_s^4 n_\theta^2 \\
&& +4\frac{\frac{1}{\rho} \partial_1  (A \hat p) R_o^2 u^4}{A_\theta} \left[ 16 \psi_{s,1}^4-\psi_{s,2} \psi_{\theta,1}^3-4 \psi_{s,1}^3 (5+2 \psi_{s,2}+4\psi_{\theta,1}) \right. \\
&& \left.\; \; \; \; \; \; \; \;  \; \; \; \; \; \; \; \;  \; \; \; \; \; \; \; \; \; \; +\psi_{s,1}\psi_{\theta,1} \left[ 3\psi_{s,2} (-3+\psi_{\theta,1})+(-5+\psi_{\theta,1}) \psi_{\theta,1} \right] \right. \\
&& \; \; \; \; \; \; \; \;  \; \; \; \; \; \; \; \;  \; \; \; \; \; \; \; \; \; \ \left. +\psi_{s,1}^2 (3+\psi_{\theta,1} (19+2\psi_{\theta,1})+\psi_{s,2} (9+6\psi_{\theta,1})) \right. \Big ] n_s^4 n_\theta^2 \\
&& +12\frac{\left( \frac{1}{\rho} \partial_1  (A \hat p) \right)^3 R_o^4}{A_\theta^2} n_s^2 n_\theta^4  \\
&& +\frac{\left( \frac{1}{\rho} \partial_1  (A \hat p) \right)^2 R_o^4}{A_\theta^2} u^2 \left[ -32 \psi_{s,1}^2-27\psi_{s,2}^2-18\psi_{s,2}\psi_{\theta,1}+\psi_{\theta,1}^2 \right. \\
&& \; \; \; \; \;  \; \; \; \; \; \; \; \; \; \;  \; \; \; \; \; \; \; \; \; \; \; \;  \; \; \; \; \; \left. +4 \psi_{s,1} (-3+18 \psi_{s,2}+4 \psi_{\theta,1}) \right. \Big ] n_s^2 n_\theta^4 \\
&& + 4 \frac{ \left(  \frac{1}{\rho} \partial_1  (A \hat p) \right)^3 R_o^6 }{A_\theta^3} n_\theta^6.
\end{array}
\]

A sufficient condition for hyperbolicity is then
\[
\begin{array}{lcl}
\psi_{s,1} > 1, &  &  \\
20 \psi_{s,1}^2+\psi_{\theta,1} (9\psi_{s,2}+5\psi_{\theta,1})-\psi_{s,1} (6+9\psi_{s,2}+19 \psi_{\theta,1}) & > & 0,\\
16 \psi_{s,1}^4-\psi_{s,2} \psi_{\theta,1}^3-4 \psi_{s,1}^3 (5+2 \psi_{s,2}+4\psi_{\theta,1}) && \\
\; \; \; \; \; \; \; \; \; \; \; \; \; \; +\psi_{s,1}\psi_{\theta,1} \left[ 3\psi_{s,2} (-3+\psi_{\theta,1})+(-5+\psi_{\theta,1}) \psi_{\theta,1} \right] && \\
\; \; \; \; \; \; \; \; \; \; \; \; \; \; \; \; \; \; \; \; \;  \; \; \; \; \; \; \; +\psi_{s,1}^2 (3+\psi_{\theta,1} (19+2\psi_{\theta,1})+\psi_{s,2} (9+6\psi_{\theta,1})) & > & 0, \\
-32 \psi_{s,1}^2-27\psi_{s,2}^2-18\psi_{s,2}\psi_{\theta,1}+\psi_{\theta,1}^2
+4 \psi_{s,1} (-3+18 \psi_{s,2}+4 \psi_{\theta,1}) & > & 0.
\end{array}
\]

We have verified that such condition is met when we use the specific profiles and parameter values in Section \ref{sec:SpecificProfiles}. Although the hyperbolicity has been proved for the special case of horizontal vessels with vanishing angular velocity, we believe this property is satisfied in a much more general context because Proposition \ref{prop:Hyp} shows that each coefficient matrix can be diagonalized.

\subsection{Specific profiles of pressure, axial and radial velocities}
\label{sec:SpecificProfiles}

Following \cite{vcanic2003mathematical}, a Hagen-Poiseuille profile is assumed for the axial velocity:
\\
\begin{equation}
\vs = e_s \vs^\star u,\:\: \vs^\star = \frac{r}{\detj}\left[1-\left(\frac{r}{R}\right)^{\gamma_s}\right],
\end{equation}
where 
\[
e_s\left(A,s,\theta\right) = \frac{A}{\int_o^R \vs^\star\left(r,R\right)\detj dr}.
\]
This profile vanishes at the artery's wall and it is strongest at the center. The exponent $\gamma_s$ controls the transition from the center to the walls. Taking $\gamma_s = 9$, we get the effect of a Newtonian fluid \cite{smith2002anatomically}. On the other hand, we use a similar profile for the angular velocity,
\begin{equation}
\vth = e_\theta \vth^\star \omega,\:\: \vth^\star = \left[1-\frac{\gamma_\theta}{\gamma_\theta + 1}\frac{r}{R}\right]\left(\frac{r}{R}\right)^{\gamma_\theta - 1},\: \text{ where }\: e_\theta\left(A,s,\theta\right) = \frac{A}{\int_o^R \vth^\star\left(r,R\right)\detj dr}.
\end{equation}
In the numerical simulations we take $\gamma_\theta = 2$. It satisfies that the linear velocity $r V_\theta$ vanishes at the center and $\partial_r (rV_\theta) =0$ at $r=R$. These gives:
\begin{align*}
\psi_{s,o} & = \: 1 - \frac{4\left(\gs +2\right)}{3\left(\gs +3\right)}\Gamma + \frac{\gs +2}{2\left(\gs +4\right)}\Gamma^2, \\
\psi_{s,1} & = \: \frac{\gs +2}{\gs +1}\left[1 - \frac{2}{3}\Gamma\right]\left[1 - \frac{2\left(2\gs +2\right)\left(\gs+2\right)}{3\left(2\gs +3\right)\left(\gs +3\right)}\Gamma\right],
\\
\psi_{s,2} & = \: \frac{\gs +2}{\gs}\left[1 - \frac{\left(\gt + 2\right)\left(2\gt + \gs +2\right)}{2\left(\gt + \gs + 1\right)\left(\gt + \gs + 2\right)}\right]\frac{1-\frac{2}{3}\Gamma}{1-\frac{2\gt+3}{2\left(\gt+3\right)}\Gamma}\times \\
&\: \left[1-\frac{2\left(\gt+\gs+1\right)\left(3\gt^2 + 2\gt\gs + 11\gt + 3\gs + 9\right)}{\left(\gt+3\right)\left(\gt+\gs+3\right)\left(3\gt +2\gs + 4\right)}\Gamma \right. \\
& \: \left. + \frac{\left(\gt+2\right)\left(\gt+\gs+1\right)\left(\gt+\gs+2\right)\left(3\gt^2 + 2\gs\gt + 15\gt + 4\gs + 16\right)}{\left(\gt+3\right)\left(\gt+4\right)\left(\gt+\gs+3\right)\left(\gt+\gs+4\right)\left(3\gt + 2\gs + 4\right)} \Gamma^2\right],
\end{align*}
\begin{align*}
\psi_{\theta,1} & = \: \frac{\gs +2}{\gs}\left[1 - \frac{\left(\gt+3\right) \left(\gt + 4\right)\left(2\gt + \gs +4\right)}{2\left(\gt+2\right)\left(\gt + \gs + 3\right)\left(\gt + \gs + 4\right)}\right]\frac{1-\frac{2}{3}\Gamma}{1-\frac{\left(\gt+3\right)\left(2\gt+5\right)}{2\left(\gt+2\right)\left(\gt+5\right)}\Gamma},\\
\psi_{\theta,2} & = \: \frac{\left(\gt+3\right)\left(\gt+4\right)\left(5\gt+6\right)}{16\left(\gt+2\right)\left(2\gt+3\right)} \frac{\left[1-\frac{2}{3}\Gamma\right]\left[1-\frac{2\left(5\gt^2 + 14\gt + 10\right)}{\left(2\gt+5\right)\left(5\gt+6\right)}\Gamma \right]}{\left[1-\frac{2\gt +3}{2\left(\gt+3\right)}\Gamma \right]\left[1-\frac{\left(\gt+3\right)\left(2\gt+5\right)}{2\left(\gt+2\right)\left(\gt+5\right)}\Gamma\right]}, \text{ and }\\
A_\theta & = \: \frac{\left(\gt+2\right)^2}{\left(\gt+3\right)\left(\gt+4\right)}\left[\frac{1-\frac{\left(\gt+3\right)\left(2\gt+5\right)}{2\left(\gt+2\right)\left(\gt+5\right)}\Gamma}{1-\frac{2\gt +3}{2\left(\gt+3\right)}\Gamma}\right]R^2.
\end{align*}

The elasticity properties of the vessel can be described by the dependance of the transmural pressure on the radius $A$, and it must be an increasing function of $A$ in order to maintain hyperbolicity. As discussed in \cite{bessonov2016methods}, the elasticity properties of the vessel's wall may be impacted by the contraction of surrounding muscles or pathologies such as aneurysms, among others. Although deriving an explicit dependance $p = p(R,s)$ is complicated, valid expressions can be found in \cite{quarteroni2004mathematical,MULLER201353,bessonov2016methods}. Following \cite{vcanic2003mathematical}, the numerical tests use the following expression of the transmural pressure, 
\begin{equation}
\label{eq:PressureYoungModulus}
\begin{array}{c}
p\left(A,s,\theta\right) = G_o\left(s,\theta\right) \left( \left(\dfrac{A}{A_o\left(s,\theta\right)}\right)^{\beta/2} - 1\right),
\end{array}
\end{equation}
where $A_o$ as defined above. This includes the effect of the wall's thickness and the stress-strain response.  Here, $G_o$ is the elasticity coefficient. Shear stress is ignored, and it is assumed that the transmural pressure of the fluid is the only force exerted on the vessel's wall. The parameter $\beta>1$ corresponds to a non-linear stress-strain response. The value $\beta=2$ provides a good approximation for experimental data \cite{vcanic2003mathematical}. The dependence of $G_o$ and $R_o$ on $s$ and $\theta$ allows us to explore the change in elasticity properties of the vessel's wall, or to explore the influence of vessel tapering on shock formation \cite{Canic2002,Formaggia2003}. Here we adopt the parametrization of the elasticity parameter in terms of the Young's modulus and wall's thickness given by \cite{xiao2014systematic}
\begin{equation}\label{eq:GoDef}
G_o(s) = \frac{4}{3} E_Y \frac{h_d}{r_d}, E_Y = \frac{3}{2} \rho \frac{r_d}{h_d} c_d^2,
\end{equation}
where $E_Y$ is the Young's modulus, $r_d$ is the radius at diastolic pressure, $h_d$ is the wall's thickness, and $c_d$ is the pulse wave speed. \\

The explicit expressions of the transmural pressure decomposition that appears in model \eqref{eq:ConservationComplete} that correspond to the transmural pressure relation \eqref{eq:PressureYoungModulus} are
\begin{equation*}
\phat\left(A,s,\theta\right) = \frac{\beta}{\beta +2}p - \frac{\beta}{\beta + 2}G_o\frac{A_o - A}{A},\;\:\:\: \pbar\left(A,s,\theta\right) = \frac{2}{\beta +2}p + \frac{\beta}{\beta + 2}G_o\frac{A_o - A}{A}.
\end{equation*}

The transmural pressure terms in the source can be written in terms of the derivative of the parameters as
\begin{equation}
\label{eq:partial_p}
\partial_2 \pbar = \frac{\pbar}{G_o}\partial_2 G_o - \frac{\phat}{A_o}\partial_2 A_o,\:\:\:\: \partial_3 \pbar = \frac{\pbar}{G_o}\partial_3 G_o - \frac{\phat}{A_o}\partial_3 A_o.
\end{equation}

Furthermore, the needed expressions to compute the source terms are given by
\begin{align*}
\int_o^R r\parts\left(\detjr\right)\left[\detjr \vs\right]^2 dr & = \: -\frac{8\left(\gs+2\right)^2}{\left(\gs+3\right)\left(2\gs+3\right)}\left(\frac{A}{R}\right)^2 R\sint\alfapps u^2, \\
\int_o^R r\parth\left(\detjr\right)\left[\detjr \vs\right]^2 dr & = \: -\frac{8\left(\gs+2\right)^2}{\left(\gs+3\right)\left(2\gs+3\right)}\left(\frac{A}{R}\right)^2 R\cost\alfaps u^2, \\
\left[\detj\partr\left(\left[\detjr\right]^2 V_s\right)\right]_{r=R} & = \: -\left(\gs+2\right)\left[1-\Gamma\right]^2 \frac{2A}{R^2} u ,\\
\left[\detj\partr\left(r^2 V_\theta\right) - 2\detj rV_\theta\right]_{r=R} & = \: -\frac{\left(\gt+1\right)\left(\gt+3\right)\left(\gt+4\right)}{4\left(\gt+2\right)} \times \\
& \;\;\;\;\;\;\;\;\;\;\;\;\;\;\;\;\;\;\;\;\; \times \left[\frac{1-\Gamma}{1-\frac{\left(\gt+3\right)\left(2\gt+5\right)}{2\left(\gt+2\right)\left(\gt+5\right)}\Gamma}\right] \frac{2A}{R^2}L.
\end{align*}

\subsection{Steady-States}
Although transient flows provide a more complete description of pulsatile blood flows, it has been shown that under certain circumstances, steady states (i.e., those independent of time) provide enough information for clinical assessment. In \cite{geers2010comparison}, a 5\% difference was found in the time-averaged wall shear stress between transient and steady states. There are, however, other clinical situations where transient flows are necessary for an accurate description of the pulsatile blood flow. In any case, our numerical scheme is constructed to accurately compute transient flows, including those near steady states.

The 2D model \eqref{eq:ConservationComplete} admits a large class of steady states that arise when a delicate balance between flux gradients and source terms occurs. Here we characterize those steady states for vessels with zero curvature ($\alfaps = 0$, or $\alpha = \alpha_o$ constant), zero viscosity ($\nu = 0$) for fluids moving in the axial direction ($\omega = 0$). In those cases, equation \eqref{eq:MainSystemNotConservation} becomes

\[
\begin{array}{lcl}
\parts\left( A \: u \right) & = & 0, \\ \\
\parts\left( \psi_{s,1} A u^2 \right) & = & -A \parts\left(\frac{p}{\rho}\right)- g A \sin(\alpha_o), \\ \\
0 & = & -A \partial_\theta \left( \frac{p}{\rho} \right).
\end{array}
\]
which implies $Au = Q_1(\theta)$ is independent of $s$. The parameter $\psi_{s,1}$ is constant in vessels with zero curvature with the profiles in Section \ref{sec:SpecificProfiles}. The second equation for the balance of momentum can be re-written as
\[
A \partial_s \left( \psi_{s,1} \frac{u^2}{2} + \frac{p}{\rho} + g z_o(s) \right) = 0,
\]
where 
\[
z_o(s) = \sin(\alpha_o) \;  s
\]
is the artery's elevation above a reference height. As a result, smooth steady states for vessels with zero curvature, zero viscosity and vanishing angular velocity satisfy that the discharge $Q_1 = A u$ and the energy $E = \psi_{s,1} \frac{u^2}{2} + \frac{p}{\rho}+gz_o(s)$ are independent of $s$, whereas  the transmural pressure $p$ is independent of $\theta$. In particular, one could have constant discharge and energy. The steady states at rest correspond to the special case
\begin{equation}
\label{eq:SteadyStateRest}
u= 0, \omega = 0, R = R_o(s,\theta), \alpha(s) = 0.
\end{equation}
Below we construct a numerical scheme that respects those steady states at rest for arteries with arbitrary cross sections. \\

\section{Central-upwind Numerical Scheme}
\label{sec:NumScheme}

In this work, we use a central-upwind scheme, whose semi-discrete formulation is obtained after integrating equation \eqref{eq:ConservationComplete} over each cell $\mathcal{C}_{j,k}:=\left[s_{j-\half},s_{j+\half}\right]\times\left[\theta_{k-\half},\theta_{k+\half}\right]$, with center at $(s_j,\theta_k)$, $s_{j\pm \half}= s_j \pm  \Delta s/2$ and $\theta_{k \pm \half}=\theta_k \pm \Delta \theta/2$. The  cell averages $\overline{\vect U}_{j,k}\left(t\right)$
\[
\overline{\vect U}_{j,k}\left(t\right) = \frac{1}{\Delta s \Delta \theta} \int_{\mathcal{C}_{j,k}}\int \vect U\left(s,\theta,t\right)dsd\theta 
\]
are approximated by solving the semi-discrete formulation
\begin{equation}\label{eq:cusemi}
\frac{d}{dt}\overline{\vect U}_{j,k}\left(t\right) = -\frac{\vect H_{j+\frac{1}{2},k}^F \left(t\right) - \vect H_{j-\frac{1}{2},k}^F\left(t\right)}{\Delta s} -\frac{\vect H_{j,k+\frac{1}{2}}^G \left(t\right) - \vect H_{j,k-\frac{1}{2}}^G\left(t\right)}{\Delta \theta} + \overline{\vect S}_{j,k}\left(t\right),
\end{equation}
with numerical fluxes $\vect H^F$ and $\vect H^G$ given by \cite{kurganov2001semidiscrete},
\begin{equation} \label{eq:NumFlux}
\begin{array}{lll}
\vect H_{j+\frac{1}{2},k}^F \left(t\right) & = \: \frac{a_{j+\frac{1}{2},k}^+ F\left(\tUjpk^-\right) - a_{j+\frac{1}{2},k}^- F\left(\tUjpk^+\right)}{a_{j+\frac{1}{2},k}^+-a_{j+\frac{1}{2},k}^-} \\
& \;\;\;\;\;\;\;\;\;\;\;\;\;\;\;\; + \frac{a_{j+\frac{1}{2},k}^+ a_{j+\frac{1}{2},k}^-  }{a_{j+\frac{1}{2},k}^+-a_{j+\frac{1}{2},k}^-} \left[ \tUjpk^+ -  \tUjpk^-\right] \\ \\
\vect H_{j,k+\frac{1}{2}}^G \left(t\right) & = \: \frac{b_{j,k+\frac{1}{2}}^+ G\left(\tUjkp^-\right) - b_{j,k+\frac{1}{2}}^- G\left(\tUjkp^+\right) }{b_{j,k+\frac{1}{2}}^+-b_{j,k+\frac{1}{2}}^-} \\
& \;\;\;\;\;\;\;\;\;\;\;\;\;\;\; + \frac{b_{j,k+\frac{1}{2}}^+ b_{j,k+\frac{1}{2}}^- }{b_{j,k+\frac{1}{2}}^+-b_{j,k+\frac{1}{2}}^-}\left[ \tUjkp^+ - \tUjkp^-\right].
\end{array}
\end{equation}
For any quantity of interest $q=q(\vect U)$, the corresponding interface values are obtained via the following piece-wise linear reconstruction
\begin{equation}
\label{eq:Reconstruction}
\begin{array}{ccccccc}
\tqjpk^- & = & \overline{q}_{j,k} & + & \frac{\Delta s}{2}\left(q_s\right)_{j,k}, \\
\tqjmk^+ & = & \overline{q}_{j,k} & - & \frac{\Delta s}{2}\left(q_s\right)_{j,k}, \\
\tqjkp^- & = & \overline{q}_{j,k} & + & \frac{\Delta \theta}{2}\left(q_\theta\right)_{j,k}, \\
\tqjkm^+ & = & \overline{q}_{j,k} & - & \frac{\Delta \theta}{2}\left(q_\theta\right)_{j,k},
\end{array}
\end{equation}
where the slopes $\left(q_s\right)_{j,k}$ and $\left(q_\theta\right)_{j,k}$ are calculated using the generalized minmod limiter
\begin{align}
\label{eq:MinModLimiter}
\left(q_s\right)_{j,k} & = \:\: \text{minmod}\left(\phi\frac{\overline{q}_{j,k} - \overline{q}_{j-1,k}}{\Delta s},\frac{\overline{q}_{j+1,k} - \overline{q}_{j-1,k}}{2\Delta s},\phi\frac{\overline{q}_{j+1,k} - \overline{q}_{j,k}}{\Delta s}\right), \\
\left(q_\theta\right)_{j,k} & = \:\: \text{minmod}\left(\phi\frac{\overline{q}_{j,k} - \overline{q}_{j,k-1}}{\Delta \theta},\frac{\overline{q}_{j,k+1} - \overline{q}_{j,k-1}}{2\Delta \theta},\phi\frac{\overline{q}_{j,k+1} - \overline{q}_{j,k}}{\Delta \theta}\right) ,
\end{align}
where
\begin{equation*}
\text{minmod}\left(z_1,z_2,\cdots\right) = 
\left\{ 
\begin{array}{cc}
\text{min}_j\{z_j\} & \text{   if}\:\:\:\: z_j>0\:\: \forall j, \\
\text{max}_j\{z_j\} & \text{   if}\:\:\:\: z_j<0\:\: \forall j, \\
0 & \text{otherwise}.
\end{array}
\right.
\end{equation*}
Here, the parameter $\phi$ is used to control the amount of numerical viscosity present in the resulting scheme.\\

The discretization of the averaged source terms
\[
\overline{\vect S}_{j,k}\left(t\right) = \frac{1}{\Delta s \Delta \theta} \int_{\mathcal{C}_{j,k}}\int S\left(\vect U\right)\left(s,\theta,t\right)dsd\theta
\]
is carried out so as to satisfy the well-balanced property. This is explained in more detail in Section \ref{sec:WellBalancePositivity}.\\

The one-sided local speeds in the $s-$ and $\theta-$directions, $a_{j+\half,k}^\pm$ and $b_{j,k+\half}^\pm$, are obtained from the largest and the smallest eigenvalues of the Jacobians $\frac{\partial F\left(\vect U\right)}{\partial \vect U}$ and $\frac{\partial G\left(\vect U\right)}{\partial \vect U}$, respectively. Using \eqref{eigen_s} and \eqref{eigen_theta}, it follows that:
\begin{subequations}\label{eq:LocalSpeed}
\begin{align}
a_{j+\half,k}^+ = &\: \text{max}\Big\{\left(\lambda_o^s\right)_{j+\half,k}^-,\left(\lambda_+^s\right)_{j+\half,k}^-,u_{j+\half,k}^-,\left(\lambda_o^s\right)_{j+\half,k}^+,\left(\lambda_+^s\right)_{j+\half,k}^+,u_{j+\half,k}^+,0 \Big\}, \\
a_{j+\half,k}^- = &\: \text{min}\Big\{\left(\lambda_o^s\right)_{j+\half,k}^-,\left(\lambda_-^s\right)_{j+\half,k}^-,u_{j+\half,k}^-,\left(\lambda_o^s\right)_{j+\half,k}^+,\left(\lambda_-^s\right)_{j+\half,k}^+,u_{j+\half,k}^+,0 \Big\}, \\
b_{j,k+\half}^+ = &\: \text{max}\Big\{\left(\lambda_o^\theta\right)_{j,k+\half}^-,\left(\lambda_+^\theta\right)_{j,k+\half}^-,\omega_{j,k+\half}^-,\left(\lambda_o^\theta\right)_{j,k+\half}^+,\left(\lambda_+^\theta\right)_{j,k+\half}^+,\omega_{j,k+\half}^+,0 \Big\}, \\
b_{j,k+\half}^- = &\: \text{min}\Big\{\left(\lambda_o^\theta\right)_{j,k+\half}^-,\left(\lambda_-^\theta\right)_{j,k+\half}^-,\omega_{j,k+\half}^-,\left(\lambda_o^\theta\right)_{j,k+\half}^+,\left(\lambda_-^\theta\right)_{j,k+\half}^+,\omega_{j,k+\half}^+,0 \Big\}.
\end{align}
\end{subequations}
The time integration of the ODE system \eqref{eq:cusemi} is done using the second-order strong stability preserving Runge-Kutta scheme \cite{GottliebShuTadmor2001}
\begin{subequations} \label{eq:RKEvolution}
\begin{align}
\vect U^{(1)} =  & \overline{\vect U}(t) +\Delta t \, \vect C[\overline{\vect U}(t)], \label{eq:euler} \\[0.1in] 
\vect U^{(2)} = & \ \half \overline{\vect U} + \half \left(\vect U^{(1)} + \Delta t \, \vect C[\vect U^{(1)}] \right), \label{eq:RK2ndstage} \\[0.1in] 
\overline{\vect U}(t+\Delta t) := & \ \vect U^{(2)},
\end{align}
\end{subequations}
with
\begin{equation*}
\vect C[\vect U(t)]_{j,k} = -\frac{\vect H_{j+\frac{1}{2},k}^F \left(t\right) - \vect H_{j-\frac{1}{2},k}^F\left(t\right)}{\Delta s} -\frac{\vect H_{j,k+\frac{1}{2}}^G \left(t\right) - \vect H_{j,k-\frac{1}{2}}^G\left(t\right)}{\Delta \theta} + \overline{\vect S}_{j,k}\left(t\right).
\end{equation*}

The Courant-Friedrichs-Lewy (CFL) condition that determines the time step $\Delta t$ is
\beq\label{eq:cfl}
\Delta t \leq \frac{1}{4}\min\left\{ \frac{\Delta s}{a},\frac{\Delta \theta}{b}\right\},
\eeq
where 
\[
a= \max_{j,k}\left\{ \max \left(a_{j + \half,k}^+, - a_{j +  \half,k}^-\right)\right\}, \text{ and } b = \max_{j,k}\left\{ \max \left(b_{j,k +  \half}^+, - b_{j,k + \half}^-\right)\right\}. 
\]

\subsection{Steady states at rest and positivity of the cross-sectional radius}
\label{sec:WellBalancePositivity}

The quantities of interest to compute the numerical flux in equation \eqref{eq:NumFlux} are reconstructed at the interfaces via equation \eqref{eq:Reconstruction} using the minmod limiter given by equation \eqref{eq:MinModLimiter}. However, the reconstructed values need to be implemented carefully in order to guarantee the well-balance property. For that end, we assume that the radius at rest $R_o\left(s,\theta\right)$ is defined at the interfaces $\left(s_{j},\theta_{k\pm\half}\right)$ and at $\left(s_{j\pm\half},\theta_{k}\right)$, and we define it at the center of each cell as
\begin{equation}
\label{eq:RoReconst}
R_{o,j,k} : = \frac{1}{4} \left[ R_o\left(s_{j-\half},\theta_k\right) + R_o\left(s_{j+\half},\theta_k\right) + R_o\left(s_j,\theta_{k-\half}\right) + R_o\left(s_j,\theta_{k+\half}\right) \right].
\end{equation}
On the other hand, the angle $\alpha(s)$ is defined at each interface point $s_{j+\half}$. The derivatives both at the center of each cell and at the interfaces are approximated via centered differences as
\begin{equation}
\label{eq:AlphaReconst}
\begin{array}{ll}
\alpha'(s_j) \approx \frac{\alpha \left( s_{j+\half} \right)-\alpha \left( s_{j-\half} \right)}{\Delta s}, \; \;
\alpha(s_j) \approx \frac{\alpha \left( s_{j+\half} \right)+\alpha \left( s_{j-\half} \right)}{2}, \\ \\
\alpha'(s_{j+\half}) \approx \frac{\alpha \left( s_{j+1} \right)-\alpha \left( s_{j} \right)}{\Delta s},
\end{array}
\end{equation}
which gives
\begin{equation*}
\begin{array}{ll}
A_{o,j+\half,k} = \frac{R_{o,j+\half,k}^2}{2}-\frac{R_{o,j+\half,k}^3}{3} \sin (\theta_k) \alpha'\left( s_{j+\half} \right), \; \; \\
A_{o,j,k+\half} = \frac{R_{o,j,k+\half}^2}{2}-\frac{R_{o,j,k+\half}^3}{3} \sin (\theta_{k+\half}) \alpha'\left( s_{j} \right).
\end{array}
\end{equation*}

In order to reconstruct $A$ at the interfaces, we define $\mathcal A = A/A_o$ and reconstruct the values $\mathcal A_{j+\half,k}^\pm$ and  $\mathcal A_{j,k+\half}^\pm$ at the cell interfaces using equation \eqref{eq:Reconstruction}. The cross-sectional area at the cell interfaces are then given by
\begin{equation}
\label{RReconst}
A_{j+\half,k}^\pm = \mathcal A_{j+\half,k}^\pm A_{o,j+\half,k}, \;\;\; A_{j,k+\half}^\pm = \mathcal A_{j,k+\half}^\pm A_{o,j,k+\half}.
\end{equation}
This way, if $R=R_o$ or equivalently $A=A_o$ at the center of each cell (as it occurs for steady states at rest), the same equality holds at the cell interfaces. Once the variable $A$ is reconstructed, this gives the reconstruction for $R$ by inverting it in terms of $A$. One also obtains the reconstruction for the parameter $\Gamma$ via the relation
\begin{equation}
\label{eq:GammaReconst}
\Gamma_{j+\half,k}^\pm = R_{j+\half,k}^\pm \sin(\theta_k)^\pm \alpha'(s_{j+\half} ), \; \; \Gamma_{j,k+\half}^\pm = R_{j,k+\half}^\pm \sin(\theta_{k+\half}) \alpha'(s_{j}).
\end{equation}
This immediately defines all the parameter functions $\psi_{s,o},\psi_{s,1},\psi_{s,2},\psi_{\theta,1},\psi_{\theta,2}$ and $A_\theta$ at the interfaces. 
The conserved variables $Q_1 = \psi_{s,o} A u$ and $Q_2= A L$ are reconstructed directly via equation \eqref{eq:Reconstruction}, from which we can recover the reconstructed values for $u=\frac{Q_1}{\psi_{s,o} A}, L = \frac{Q_2}{A}$, and $\omega = \frac{L}{A_\theta}$.

The source terms in equation \eqref{eq:SourceGeneral} do not involve derivatives of the solution itself and one can use the cell averages to discretize them. The partial derivatives $\partial_2 \bar p$ and $\partial_3 \bar p$ are with respect to the explicit dependance of the fixed parameters involved in the definition of the transmural pressure. For instance, the transmural pressure $p$ given by equation \eqref{eq:PressureYoungModulus} involves the radius at rest $R_o(s,\theta)$ and the parameter $G_o(s,\theta)$. The parameter $G_o\left(s,\theta\right)$ is defined at the interfaces $\left(s_{j},\theta_{k\pm\half}\right)$ and $\left(s_{j\pm\half},\theta_{k}\right)$, and we define it at the center of each cell as
\begin{equation}
\label{eq:GoReconst}
G_{o,j,k} : = \frac{1}{4} \left[ G_o\left(s_{j-\half},\theta_k\right) + G_o\left(s_{j+\half},\theta_k\right) + G_o\left(s_j,\theta_{k-\half}\right) + G_o\left(s_j,\theta_{k+\half}\right) \right].
\end{equation}
The terms $\partial_2 \bar p$ and $\partial_3 \bar p$ are given explicitly by equation \eqref{eq:partial_p}. In that case, one only needs partial derivatives of $G_o$ and $A_o$ which are approximated as
\begin{equation}
\label{eq:partialG0}
\begin{array}{l}
\partial_2 G_o (s_j,\theta_k) \approx \frac{G_o(s_{j+\half,k})-G_o(s_{j-\half,k})}{\Delta s},  \; \; \; \\ \\
\partial_3 G_o (s_j,\theta_k) \approx \frac{G_o(s_{j,k+\half})-G_o(s_{j,k-\half})}{\Delta \theta},
\end{array}
\end{equation}
\begin{equation}
\label{eq:partialA0}
\begin{array}{l}
\partial_2 A_o (s_j,\theta_k) \approx \frac{A_o(s_{j+\half,k})-A_o(s_{j-\half,k})}{\Delta s},  \; \; \;  \\ \\
\partial_3 A_o (s_j,\theta_k) \approx \frac{A_o(s_{j,k+\half})-A_o(s_{j,k-\half})}{\Delta \theta}.
\end{array}
\end{equation}

In a steady state at rest given by equation \eqref{eq:SteadyStateRest}, the reconstructed values of $u$ and $\omega$ are zero, and the equality $R=R_o$ holds at the interfaces. As a result, all the numerical fluxes $\vect H_{j+\half,k}^F, \vect H_{j,k+\half}^G$ and the source terms $\overline{\vect S}_{j,k}$ vanish. We have proved the following proposition.\\

\begin{proposition}
Consider system \eqref{eq:ConservationComplete}, \eqref{eq:ConservedQFlux} and \eqref{eq:SourceGeneral} with transmural pressure given by \eqref{eq:PressureYoungModulus}. Then, the numerical scheme \eqref{eq:NumFlux}-\eqref{eq:cfl} with the discretizations given by \eqref{eq:RoReconst}-\eqref{eq:partialA0} is well-balanced, i.e., $\overline{\vect U}(t+\Delta t) = \overline{\vect U}(t)$ for steady states at rest.
\end{proposition}

The following proposition shows that the CFL condition \eqref{eq:cfl} guarantees the positivity of $A$ when the solution is computed with the Runge-Kutta method \eqref{eq:RKEvolution} and a slight modification is applied to the reconstruction at the interfaces. This is particularly important in situations where the cross section is small. This is not a relevant case from the medical point of view. However, we present it here for the sake of completeness and it would be useful for applications involving collapsed tubes. 

Since we have reconstructed $\mathcal A = A/A_o$ at the interfaces, 
\begin{equation}
\label{eq:AveragesRelation}
\overline{A}_{j,k} = \frac{1}{4}\left[A_{j-\half,k}^+ + A_{j+\half,k}^- + A_{j,k+\half}^- + A_{j,k-\half}^+\right]
\end{equation}
does not necessarily hold unless $A_o$ is constant. In the cases where one decides to implement the positivity-preserving property, a modification in the reconstruction must be implemented. Namely,
\begin{subequations}
\begin{align}
\label{eq:Correction}
\mathcal{A}_{j+\half,k}^- = &\: \overline{\mathcal{A}}_{j,k} + \frac{\Delta s}{2}\left(\mathcal{A}_s\right)_{j,k} \\
A_{j+\half,k}^- = &\: \min\left( \max\{  \mathcal{A}_{j+\half,k}^- A_{o,j+\half,k},A_{th} \}, 2  \bar A_{j,k} \right) \label{eq:CorrectionB} \\
A_{j-\half,k}^+ = &\: 2 \bar A_{j,k} - A_{j+\half,k}^- \label{eq:CorrectionC}
\end{align}
\end{subequations}
with the analogous procedure for $A_{j,k\pm \half}^\pm$. Here $A_{th}$ is a threshold value needed to maintain positivity in the reconstruction of the interface values. Equation \eqref{eq:CorrectionB} guarantees that $0 \le A_{j+\half,k}^- \le 2 \bar A_{j,k}$, and the discretization is consistent with the equations. Equation \eqref{eq:CorrectionC} guarantees $A_{j-\half,k}^+ \ge 0$. As a result, the above corrections ensure the positivity of the reconstructed values as well as the relation in equation \eqref{eq:AveragesRelation} needed in the proof below. It is also important to mention that the above corrections does not affect the well-balance property.

\begin{proposition}\label{prop:Positivity}
Consider the scheme with the reconstruction algorithm described in $\mathcal{x}$ \ref{sec:WellBalancePositivity}. If the cell averages $\overline{A}_{j,k}\left(t\right)$ are such that
\[
\overline{A}_{j,k}\left(t\right) \geq 0\:\:\: \forall\:\: j,k,
\]
then the cell averages $\overline{A}_{j,k}\left(t+\Delta t\right)$ computed with the Runge-Kutta method \eqref{eq:RKEvolution} for all $j,k$, under the CFL limitation \eqref{eq:cfl}, will yield
\[
\overline{A}_{j,k}\left(t+\Delta t\right) \geq 0\:\:\: \forall\:\: j,k,
\]
\end{proposition}
\begin{proof}
Since our Runge-Kutta numerical scheme can be written as a convex combination of Euler steps, one only needs to prove it for just one forward Euler step. The first component in equation \eqref{eq:cusemi} can be written as
\[
\begin{array}{lll}
\overline{A}_{j,k}\left(t+\Delta t\right) & = &  \overline{A}_{j,k}\left(t\right) - \frac{\Delta t}{\Delta s}\left[\left(\vect H_{j+\half,k}^F\right)^{(1)} - \left(\vect H_{j-\half,k}^F\right)^{(1)}\right] \\
&&- \frac{\Delta t}{\Delta \theta}\left[\left(\vect H_{j,k+\half}^G\right)^{(1)} - \left(\vect H_{j,k-\half}^G\right)^{(1)}\right].
\end{array}
\]
Using \eqref{eq:NumFlux}, we can rewrite it as
\begin{eqnarray*}
\overline{A}_{j,k}\left(t+\Delta t\right) & = & \left[\frac{1}{4} - \frac{\Delta t}{\Delta s}a_{j+\half,k}^+\frac{u_{j+\half,k}^- - a_{j+\half,k}^-}{a_{j+\half,k}^+ - a_{j+\half,k}^-}\right]A_{j+\half,k}^- \\
&& + \left[\frac{1}{4} + \frac{\Delta t}{\Delta s}a_{j-\half,k}^-\frac{a_{j-\half,k}^+ - u_{j-\half,k}^+}{a_{j-\half,k}^+ - a_{j-\half,k}^-}\right]A_{j-\half,k}^+\\
&& + \left[\frac{1}{4} - \frac{\Delta t}{\Delta \theta}b_{j,k+\half}^+\frac{\omega_{j,k+\half}^- - b_{j,k+\half}^-}{b_{j,k+\half}^+ - b_{j,k+\half}^-}\right]A_{j,k+\half}^- \\
&& + \left[\frac{1}{4} + \frac{\Delta t}{\Delta \theta}b_{j,k-\half}^-\frac{b_{j,k-\half}^+ - \omega_{j,k-\half}^+}{b_{j,k-\half}^+ - b_{j,k-\half}^-}\right]A_{j,k-\half}^+\\
& & -\frac{\Delta t}{\Delta s}a_{j+\half,k}^- \frac{a_{j+\half,k}^+ - u_{j+\half,k}^+}{a_{j+\half,k}^+ - a_{j+\half,k}^-}A_{j+\half,k}^+ \\
&& +\frac{\Delta t}{\Delta s}a_{j-\half,k}^+ \frac{u_{j-\half,k}^- - a_{j-\half,k}^-}{a_{j-\half,k}^+ - a_{j-\half,k}^-}A_{j-\half,k}^-\\
& &  -\frac{\Delta t}{\Delta \theta}b_{j,k+\half}^- \frac{b_{j,k+\half}^+ - \omega_{j,k+\half}^+}{b_{j,k+\half}^+ - b_{j,k+\half}^-}A_{j,k+\half}^+ \\
&& +\frac{\Delta t}{\Delta \theta}b_{j,k-\half}^+ \frac{\omega_{j,k-\half}^- - b_{j,k-\half}^-}{b_{j,k-\half}^+ - b_{j,k-\half}^-}A_{j,k-\half}^- \\
& & + \overline{A}_{j,k} - \frac{1}{4}\left[A_{j-\half,k}^+ + A_{j+\half,k}^- + A_{j,k+\half}^- + A_{j,k-\half}^+\right].
\end{eqnarray*}
The first four terms in the above equation will be non-negative under the CFL restriction \eqref{eq:cfl}. Also, since $a_{j+\half,k}^-\leq 0$, $b_{j,k+\half}^-\leq 0$, $a_{j-\half,k}^+\geq 0$ and $b_{j,k-\half}^+\geq 0$, the following four terms in the above equation are also non-negative.

\end{proof}

\section{Numerical Experiments }
\label{sec:NumTests}

Different numerical experiments are presented to show the merits of the numerical scheme and the dynamics of the flow given by the model derived in this paper. One can analyze situations where vessels exhibit non-uniform elasticity properties and the effect on the dynamics of the flow.

The velocity field in 3D views of the artery is computed as follows. First, we need to compute the curvature radius, which is given by
\[
R_c = \frac{\left(R^2 + R_\theta^2\right)^{\frac{3}{2}}}{R^2 + 2R_\theta^2 - R R_{\theta \theta}} = \frac{R\left(1 + \left(\frac{R_\theta}{R}\right)^2\right)^{\frac{3}{2}}}{1 + 2\left(\frac{R_\theta}{R}\right)^2 - \frac{R_{\theta \theta}}{R}},\:\: 
\]
and we also define 
\[
R_\theta = \parth R,\:\:\: R_{\theta\theta} = \partial^2_{\theta \theta} R.
\]
The total velocity at each point 
\[
\Big( x(s,\theta),y(s,\theta),z(s,\theta) \Big) = \Big( -R \sina\sint + x_o\left(s\right) \; , \;  R\cost \; , \;  R\cosa\sint + z_o\left(s\right) \Big)
\]
is given by
\begin{equation}
\label{eq:3DVField}
\vect{V}\left(s,\theta\right) = \left(
\begin{array}{c}
\\ \cosa \\ \\ 0 \\ \\ \sina \\ \\ 
\end{array}
\right)u + \frac{\frac{R_c}{R}}{\left(1 + \left(\frac{R_\theta}{R}\right)^2\right)^{\frac{1}{2}}} \left(
\begin{array}{c}
\\ -\sina\Big( \cost + \frac{R_\theta}{R} \sint\Big) \\ \\ -\sint + \frac{R_\theta}{R} \cost \\ \\ \cosa\Big( \cost + \frac{R_\theta}{R} \sint\Big),  \\ \\
\end{array}
\right)U_{Tang}
\end{equation}
where
\begin{equation}
\label{eq:UTang}
U_{Tang} = \frac{1}{A}\int_o^R r\vth\detj dr.
\end{equation}
For convenience and ease of notation, we all define $\mathcal{R}=R/R_o$ at the center of each cell. \\

In all cases, we apply periodic boundary conditions in $\theta$. The boundary conditions in the axial directions are specified in each numerical test. Also, the parameters used for these numerical experiments are: blood density $\rho = 1050 \text{ kg m}^{-3}$, blood viscosity $\nu = 4 \text{ cP}$, $\beta = 2$, $\gs = 9$ and $\gt = 2$.

\subsection{Horizontal vessel with tapering: evolution of perturbation}

In this first numerical test, we consider the simple case of a horizontal vessel ($\alfas = 0$) with tapering. That is, $A_o$ is given by
\[
A_o\left(s,\theta\right) = A_o^\star\left(1-sT_p\right),\:\:\:  R_o\left(s,\theta\right) = \sqrt{2A_o\left(s,\theta\right)\:},
\]  
where $T_p=0.005 \text{ cm}^{-1}$ is the tapering factor and $A_o^\star = \left(0.82\text{ cm}\right)^2$. The initial vessel's radius consist of a perturbation from a steady state. The perturbation is located in the middle of the artery. Specifically, the center is at $s^\star=25 $ cm and $\theta^\star = \frac{\pi}{4}$ rad. The initial radius is then given by
\[
R\left(0,s,\theta\right) = \mathcal{R}\left(0,s,\theta\right)R_o\left(s,\theta\right), \text{ where }
\]
\[
\mathcal{R}\left(0,s,\theta\right) = \left\{
\begin{array}{ccc}
1 & \text{  if  } & \frac{d\left(s,\theta\right)}{R_o\left(s^\star,\theta^\star\right)} > 1  \\
&& \\
1 + \frac{1}{5}\sin\left(\left[1-\frac{d\left(s,\theta\right)}{R_o\left(s^\star,\theta^\star\right)}\right]\frac{\pi}{2}\right) & \text{  if  } & \frac{d\left(s,\theta\right)}{R_o\left(s^\star,\theta^\star\right)} \leq  1
\end{array}
\right.,
\]
and 
$$
d\left(s,\theta\right) = \sqrt{\frac{1}{4}\left[x\left(s^\star,\theta^\star\right) - x\left(s,\theta\right)\right]^2 + \left[y\left(s^\star,\theta^\star\right) - y\left(s,\theta\right)\right]^2 + \left[z\left(s^\star,\theta^\star\right) - z\left(s,\theta\right)\right]^2}.
$$
Neumann boundary conditions are imposed at both ends ($s=0,s_L$) with $s_L=50 \text{ cm}$.

\begin{figure}[h!]
\begin{center}
{
\includegraphics[width=0.95\textwidth]{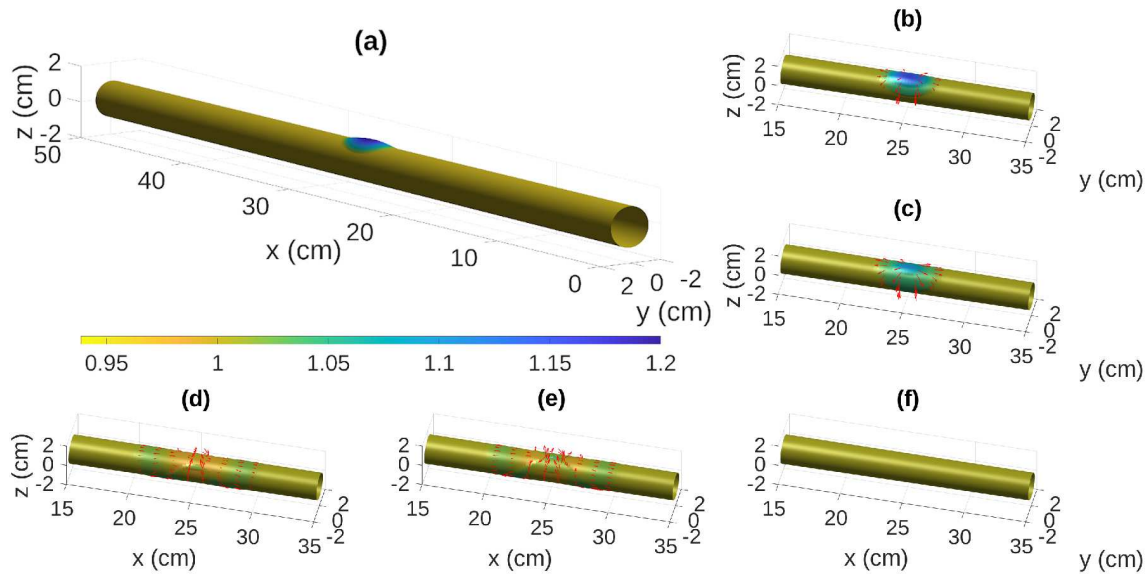}
}
\end{center}
\caption{\label{fig:horizontalbf} A 3D view of the vessel. The arrows near the wall indicate the velocity field given by equation \eqref{eq:3DVField}. The color bar indicates the ratio of the vessel radius at time $t$ and its initial value ($R(t\; ; s,\theta)/R_o(s, \theta)$). The initial conditions, which consists of a radius perturbation near $s^\star=25 \text{ cm}$ and $\theta^\star = \pi/4$ rad is shown in (a), where the entire artery is visualized. The rest of the panels show a section $15 \text{ cm } \le s \le 35 \text{ cm}$ where the perturbation evolves at times $t= 0.0005 \text{ s}$ (b), $t= 0.001 \text{ s}$ (c), $t= 0.0045 \text{ s}$ (d),  $t= 0.005 \text{ s}$ (e) and  $t= 0.1 \text{ s}$ (f). The arrows indicate the 3D velocity field given by equation \eqref{eq:3DVField}. }
\end{figure} 

In this first numerical test, the initial conditions consist of a radius perturbation from a steady state. That is, the transmural pressure is zero everywhere in the artery, except in an area near $(s^\star,\theta^\star)$ where the radius is above the steady state one and the transmural pressure becomes positive. Figure \ref{fig:horizontalbf} shows a 3D view of the artery with the above initial conditions in panel (a). This generates a displacement that consist of a radial expansion at early times, and it can be observed in panels (b) and (c). The colobar indicates the ratio of the vessel radius at time $t$ and its initial value ($R(t\; ; s,\theta)/R_o(s, \theta)$), which can help us identify the evolution of the perturbation. The initial perturbation covers only a partial side of the artery's wall and the displacement goes in both the axial and angular directions. At later times in panels (d) and (e), the displacement has already reached the opposite side of the wall and it has come back to the initial location by periodicity in the angular direction. The last panel (f) shows the solution at time $t=0.1 \text{ s}$ where the displacement has already propagated in the axial direction outside the visualized region. As a result, the artery has recovered its initial unperturbed steady state.

\subsection{Aorta vessel with discharge}

\label{sec:AortaCC}

 \begin{table}
\begin{center}
\vline
\begin{tabular}{llllll}
\hline
Segment & \vline Length & \vline  Left radius & \vline  Right radius  & \vline Left speed  &  \vline Right speed  \\ 
 & \vline $\text{ cm}$ & \vline $\text{ cm}$  & \vline $\text{ cm}$  & \vline $\text{ ms}^{-1}$ &  \vline $\text{ ms}^{-1}$ \\ 
\hline
I & \vline  7.0357  & \vline 1.52 & \vline  1.39   & \vline 4.77 & \vline  4.91 \\ 
\hline
II & \vline  0.8  & \vline 1.39 & \vline  1.37  & \vline 4.91 & \vline  4.93 \\ 
\hline
III & \vline  0.9  & \vline 1.37 & \vline  1.35   & \vline 4.93 & \vline  4.94 \\ 
\hline
IV & \vline  6.4737  & \vline 1.35 & \vline  1.23  & \vline 4.94 & \vline  5.09 \\ 
\hline
V & \vline  15.2  & \vline 1.23 & \vline  0.99  & \vline 5.09 & \vline  5.43 \\ 
\hline
VI & \vline  1.8  & \vline 0.99 & \vline  0.97  & \vline 5.43 & \vline  5.46 \\ 
\hline
VII & \vline  0.7  & \vline 0.97 & \vline  0.962  & \vline 5.46 & \vline  5.48 \\ 
\hline
VIII & \vline  0.7  & \vline 0.962 & \vline  0.955   & \vline 5.48 & \vline  5.49 \\ 
\hline
IX & \vline  4.3  & \vline 0.955 & \vline  0.907  & \vline 5.49 & \vline  5.57 \\ 
\hline
X & \vline  4.3  & \vline 0.907 & \vline  0.86  & \vline 5.57 & \vline  5.66 \\ 
\hline
\end{tabular}
\hspace{-7 pt} \vline
\end{center}
\caption{\label{table:Data} Description of aorta's geometry and dimensions. The segments, their lengths, left and right radii are shown in the first four columns. The last two columns show the left and right pulse wave speeds. Such values were obtained from \cite[Table IV]{xiao2014systematic}.}
\end{table}

\begin{figure}[h!]
\begin{center}
{\includegraphics[width=0.8\textwidth]{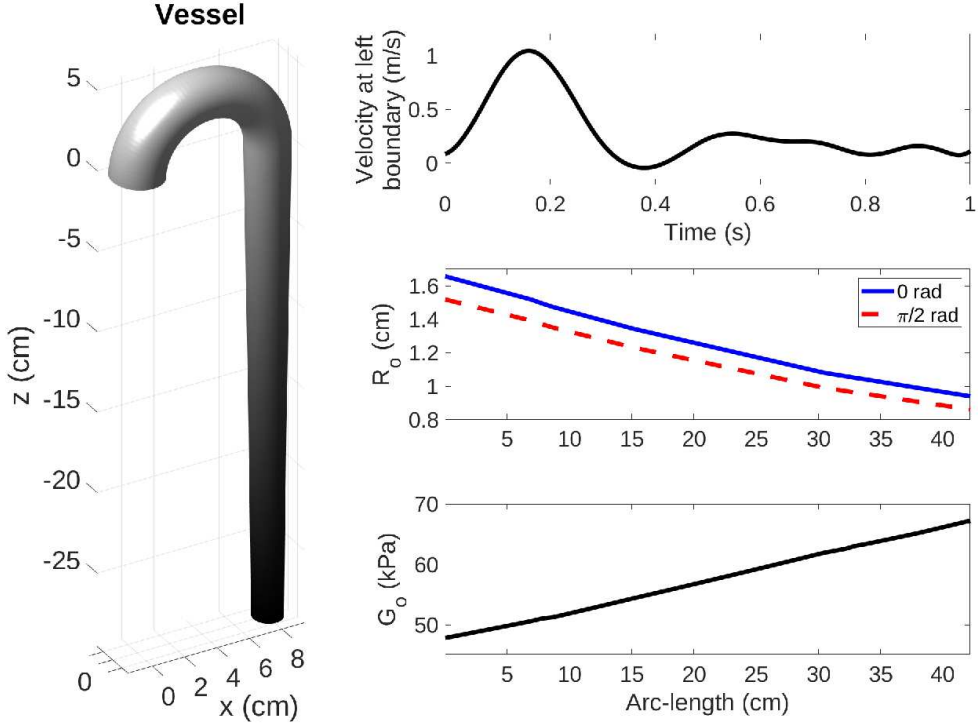}}
\end{center}
\caption{\label{fig:IC_ALL} Blood flow simulation passing through the full aorta. Left panel: 3D view of the aorta. Top right: Velocity at the left boundary in a cardiac cycle as a function of time. Middle right: Profile of $R_o$ as a function of arc-length $s$ for different angles. Bottom right: Profile of $G_o$ as function of arc-length $s$. }
\end{figure}

The previous case showed that the model and the numerical scheme produces good results for perturbations to steady states in horizontal vessels. For the rest of the numerical tests, we will consider geometries similar to an idealized aorta without branches. 

Let $R_o^\star\left(s\right)$ be the piecewise linear function of $s$ obtained according to the radius at diastolic pressure shown in \cite[Table IV]{xiao2014systematic}, which we present in table \ref{table:Data} for the convenience of the reader. The initial conditions for the artery's geometry is described by cross sections given by
\[
R_o\left(s,\theta\right) = R_o^\star\left(s\right) h\left(\theta\right),
\]
where $h(\theta)$ is a function that determines the type of cross section that we may have. We note that we obtain circular cross sections when $h\left(\theta\right)$ is constant. As it is reported in \cite{kyriakou2020analysing}, however, we may observe cross sections that have an elliptical-like shape in the aorta. We then choose
\[
h\left(\theta\right) = \sqrt{\frac{1- \xi^2 \sint^2}{1-\xi^2}},
\]
where $\xi$ is the eccentricity and $\xi\in\left[0,1\right)$. Graphs of $R_o$ as a function of axial position $s$ for different values of $\theta$ are displayed in the middle right panel of Figure \ref{fig:IC_ALL} .

The parameter $G_o^\star\left(s\right)$ is given by equation \eqref{eq:GoDef}, where $c_d$ is a piecewise linear function according to the velocity at diastolic pressure shown in Table \ref{table:Data} and in \cite[Table IV]{xiao2014systematic}. A graph of $G_o$ is displayed in the bottom right panel of Figure \ref{fig:IC_ALL}. As an approximation to the aorta's curvature, the angle $\alfas$ is given by
\[
\alfas = \left\{
\begin{array}{ccc}
\left[1 - \frac{s}{12.63\text{ cm} }\right] \frac{\pi}{2} & \text{  if  } & 0\leq s \leq 12.63\text{ cm},  \\
& & \\
-\frac{\pi}{2} & \text{  if  } & s> 12.63\text{ cm} .
\end{array}
\right.
\]
That is, the vessel is straigth up at the upstream boundary $\alpha(0 \text{ cm}) = \pi/2$ and points down at $s=12.63 \text{ cm}$, where $\alpha(12.63 \text{ cm}) = -\pi/2$. Figure \ref{fig:IC_ALL} (left panel) shows a 3D view of the tapered vessel at time $t=0$ s. Here, we use 200 grid points in the axial direction and 180 grid points in the angular direction. The initial condition is given by $A=A_o$, $u=0$ and $L = 0$ in a tilted vessel with elliptical geometric shape ($\xi = 0.4$), which would have corresponded to a steady state if the vessel was horizontal. \\

At the left boundary ($s=0$), we impose a velocity that corresponds to a cardiac cycle (to be specified below) and Dirichlet boundary conditions for the radius $R=R_o$ and $\omega=0$. The discharge at the left boundary breaks the balance and induces a moving state. Neumann boundary conditions are imposed at the right boundary. The time series for the velocity at the left boundary was obtained from \cite{Canic2002}, and it was approximated using the first 15 elements of its Fourier decomposition. Initially, the velocity at the left boundary increases up to speeds above 1 $\text{ms}^{-1}$. A graph of the inlet velocity as a function of time can be found in the top right panel of Figure \ref{fig:IC_ALL} . 

\begin{figure}[h!]
\begin{center}
{\includegraphics[width=\textwidth]{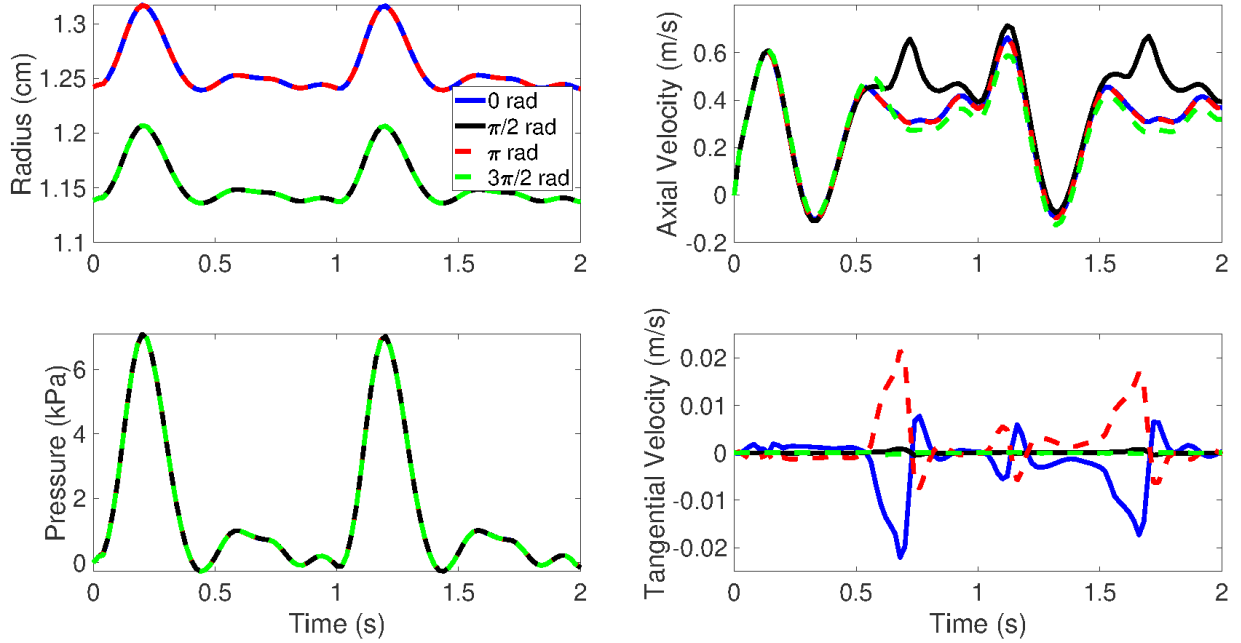}}
\end{center}
\caption{\label{fig:CC_Profiles} Profiles as a function of time at $s=21.10$ cm. Top left: Radius $R$ at different angles. Top right: Axial velocity $u$. Bottom left: Pressure $p$. Bottom right: Tangential velocity $U_{Tang}$.}
\end{figure}

In Figure \ref{fig:CC_Profiles}, we show the evolution of four variables, radius $R$, axial velocity $u$, pressure $p$ and tangential velocity $U_{Tang}$, over 2 seconds at $s = 21.10$ cm. Here, the vessel's radius $R$ is increased due to the influence of the inlet velocity given by the cardiac cycle. On the hand, the transmural pressure reaches its maximum of approximately $7.1$ kPa near $t=0.2$ seconds, followed by a decay. Here, the transmural pressure profile is similar for each $\theta$. In the top right panel we observe the evolution of the axial velocity. Since the initial condition is $u=0$, the profile starts with an increment to $0.61\text{ ms}^{-1}$, followed by a decay to $-0.1\text{ ms}^{-1}$, and by an increment to $0.44 \text{ ms}^{-1}$ reached at $0.52 \text{ s}$. After this time, the profiles given by different $\theta$ values diverge from each other. At $\theta=\pi/2 \text{ rad}$ the velocity reached a maximum of approximately $0.69 \text{ ms}^{-1}$ at $t=0.70$ seconds, while the other profiles show a decay to $0.31 \text{ ms}^{-1}$ at the same time. After that, the profiles evolve in a quasi-periodic way.\\

Finally, the tangential velocity profiles are shown in the bottom right panel of Figure \ref{fig:CC_Profiles}. Here, we observe that the value is much lower than the axial velocity. One can observe that the tangential velocities are much weaker at angles $\theta=\pi/2$ and $\theta=3\pi/2$ rad, while the other two ($\theta=0$ and $\theta=\pi$ rad) are approximately opposite in sign.

\subsection{Idealized aorta vessel with a bulge}

\begin{figure}[h!]
\begin{center}
{\includegraphics[width=\textwidth]{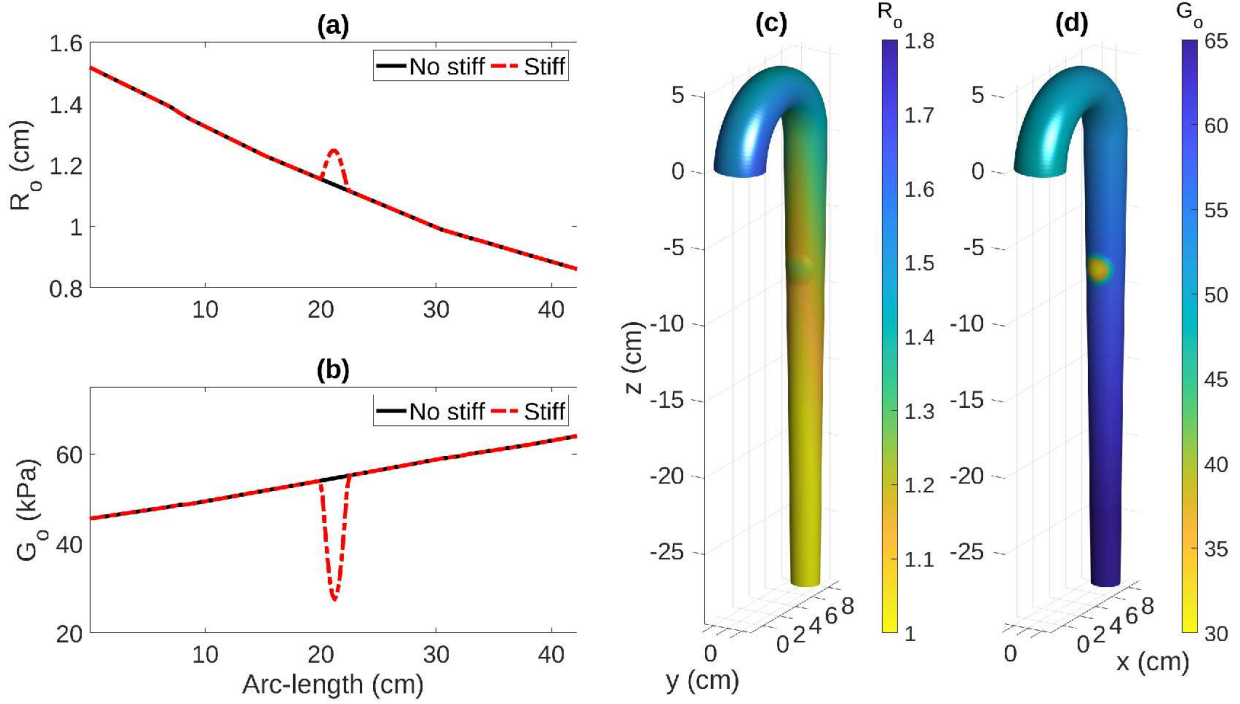}}
\end{center}
\caption{\label{fig:Ex2_Initial_Ro_Go} Parameters for an artery with a bulge in a localized region where the wall is less rigid. In black solid line, profiles of $G_o$ (panel (a)) and $R_o$  (panel (b)) at $\theta = 3\pi/2 \text{ rad}$ are shown, while the red dotted lines correspond to the base case with no bulging. The profiles are given by equations \eqref{eq:BulgeG0} and \eqref{eq:BulgeR0}. Panel (c): 3D visualization of an idealized aorta where the color bar denotes $G_o$ values. Panel (d): same a in middle panel with $R_o$ values in the color bar.}
\end{figure}

Non-uniform elasticity properties in a vessel may be caused by diseases such as stenosis and aneurysms. In this numerical test, we analyze possible changes in the flow dynamics when the parameters $G_o$ and $R_o$ are non-uniform in a localized regions in the artery's wall. In particular, $G_o$ is reduced at $(s= s^\star = 21 \text{ cm},\theta = \theta^\star = 3\pi/2 )$ and $R_o$ is increased near that point, when compared to the previous case. Such changes in the parameters are aimed at simulating an artery with a bulge in a localized region where the artery is also less rigid. \\

The parameters $G_o$ and $R_o$ are given by
\begin{equation}
\label{eq:BulgeG0}
G_o\left(s,\theta\right) = \left\{
\begin{array}{ccc}
&& \\
G^\star_o\left(s\right) & \text{  if  } & d\left(s,\theta\right) > h\left(\theta\right) R^\star_o\left(s\right), \\
&& \\
\left[1 - \frac{1}{2}\sin\left(\left[1-\frac{d\left(s,\theta\right)}{h\left(\theta\right) R_o^\star\left(s\right)}\right]\frac{\pi}{2}\right)\right] G^\star_o\left(s\right) & \text{  if  } & d\left(s,\theta\right) \leq h\left(\theta\right) R^\star_o\left(s\right), \\
&&
\end{array}
\right.
\end{equation}
and
\begin{equation}
\label{eq:BulgeR0}
R_o\left(s,\theta\right) = \left\{
\begin{array}{ccc}
&& \\
h\left(\theta\right) R^\star_o\left(s\right) & \text{  if  } & d\left(s,\theta\right) > h\left(\theta\right) R^\star_o\left(s\right), \\
&& \\
\left[1 + \frac{1}{5}\sin\left(\left[1-\frac{d\left(s,\theta\right)}{h\left(\theta\right) R^\star_o\left(s\right)}\right]\frac{\pi}{2}\right)\right]h\left(\theta\right) R^\star_o\left(s\right) & \text{  if  } & d\left(s,\theta\right) \leq h\left(\theta\right) R^\star_o\left(s\right), \\
&&
\end{array}
\right.
\end{equation}
where 
$$
d\left(s,\theta\right) = \sqrt{\frac{1}{4}\left[x\left(s^\star,\theta^\star\right) - x\left(s,\theta\right)\right]^2 + \left[y\left(s^\star,\theta^\star\right) - y\left(s,\theta\right)\right]^2 + \left[z\left(s^\star,\theta^\star\right) - z\left(s,\theta\right)\right]^2}.
$$
A graph of the parameters at $\theta = 3\pi/2 \text{ rad}$ as a function of $s$ is shown in panels (a) and (b) in Figure \ref{fig:Ex2_Initial_Ro_Go} . A 3D visualization of the aorta using $R_o$ (c) and $G_o$ (d) in the color bar are displayed to localize the region where the elasticity properties of the artery vary. 

\begin{figure}[h!]
\begin{center}
{\includegraphics[width= \textwidth]{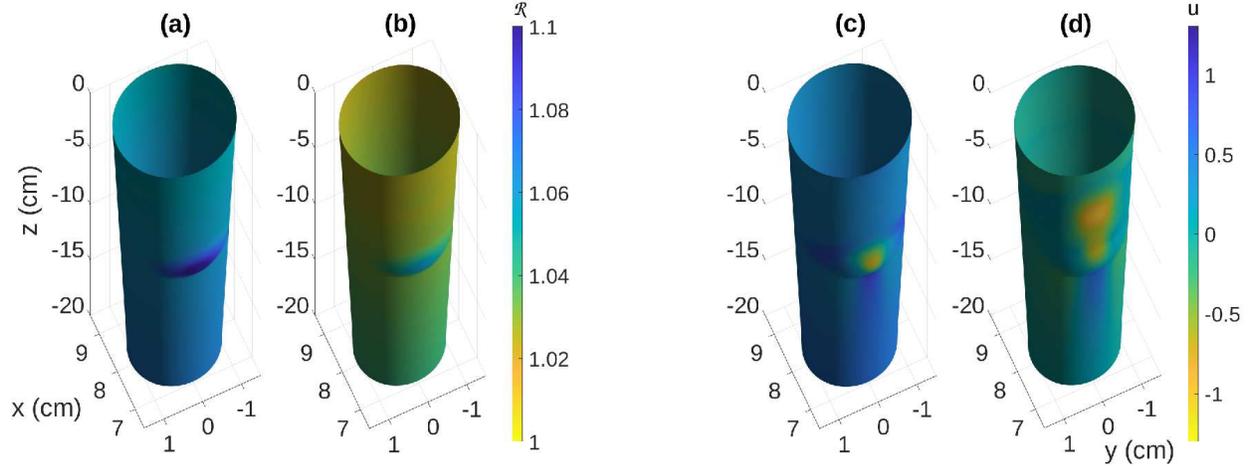}}
\end{center}
\caption{\label{fig:Ex2_R_Us} Panels (a) and (b): 3D view of the artery with color bar computed based on $\mathcal{R}=R/R_o$ at times $t=0.2\text{ s}$ and $t = 0.3 \text{ s}$ respectively. Panels (c) and (d): Same as in (a) and (b) with a color bar computed based on $u$.}
\end{figure}

As initial conditions, we set
\[
R\left(0,s,\theta\right)=R_o\left(s,\theta\right),\:\:\:\: u\left(0,s,\theta\right)=0\:\:\:\text{ and }\:\:\: L\left(0,s,\theta\right)=0.
\]

Figure \ref{fig:Ex2_R_Us}  shows the effect of the bulge with non-uniform elasticity parameters in the flow dynamics. For instance, panels (a) and (b) shows a 3D view of the artery near the bulge where the color bar indicates the ratio $R/R_o$ at times $t=0.2 \text{ s}$ and $t=0.3 \text{ s}$ respectively. Such ratio indicates how much the artery's radius has been deformed from the initial conditions. One can observe a stronger deviation from the initial conditions (about 10\%) near the bulge at $t=0.2 \text{ s}$, when compared to the rest of the artery. Such deviation is reduced at $t=0.3\text{ s}$. On the other hand, a color bar computed based on the axial velocity is displayed in panels (c) and (d). We observe a negative displacement in the upper side of the bulge and a positive displacement in the lower side of the bulge at $t=0.3 \text{ s}$. However, the axial velocity becomes positive everywhere at later times (not shown) due to gravity and the fluid discharge in the upstream boundary. 

\begin{figure}[h!]
\begin{center}
{\includegraphics[width=\textwidth]{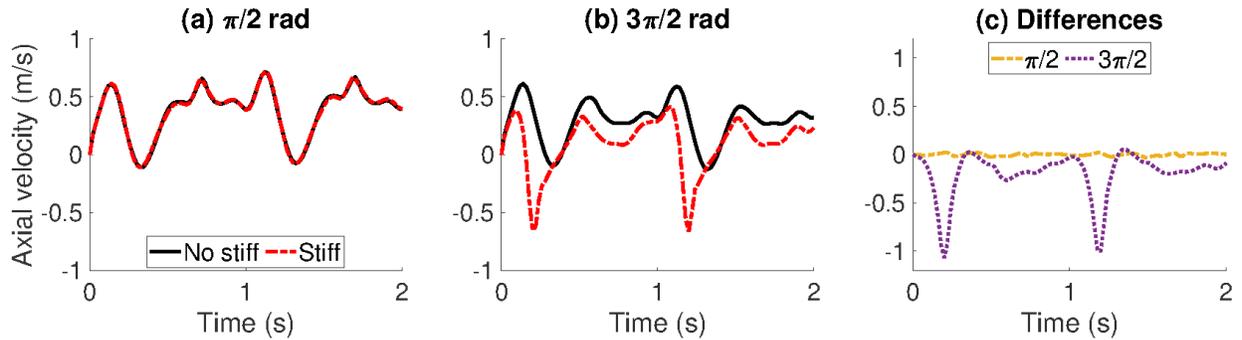}}
\end{center}
\caption{\label{fig:Ex2_Velocities} Axial velocity profiles as a function of time at $s=21 \text{ cm}$. Panel (a): Profiles at $\theta = \pi/2$ rad. Panel (b): Profiles at $\theta = 3\pi/2 $ rad. Panel (c): Profile differences for $\theta = \pi/2$ rad (ocher dashed line) and $\theta = 3\pi/2$ rad (purple dotted line).}
\end{figure}

Figure \ref{fig:Ex2_Velocities} exhibits a quantification of the observations discussed in Figure \ref{fig:Ex2_R_Us}. Specifically, the axial velocity as a function of time at $s=s^\star$, $\theta = \pi/2$ and $\theta = 3\pi/2 \text{ rad}$ are displayed in panels (a) and (b) respectively. For comparison, we include a graph corresponding to the base case in Section \ref{sec:AortaCC} . The bulge is at $\theta = 3\pi/2 \text{ rad}$ (panel (b)), where the velocity is decreasing near it. The differences could be up to about $1 \text{ ms}^{-1}$, as we see it in panel (c). Panel (a) shows the axial velocities experienced by the fluid on the opposite side of the wall, where the impact of the bulge does not seem to be significant in the time window considered here.

\subsection{Vortex-like structure in aorta vessel with a bulge}
\begin{figure}[h!]
\begin{center}
{\includegraphics[width=\textwidth]{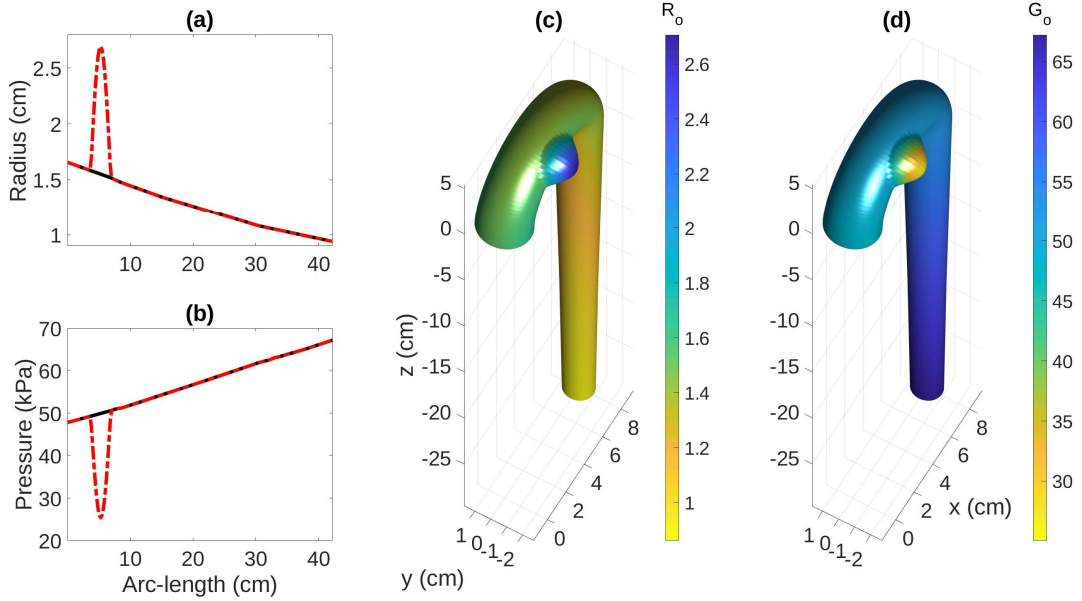}}
\end{center}
\caption{\label{fig:Ex5_Profiles} Parameters $R_o$ and $G_o$ given by equations  \eqref{Ex5:G_o} and \eqref{Ex5:R_o} as a function of $s$ at $\theta = \pi \text{ rad}$ are shown in panels (a) and (b) respectively (red dotted lines). The black solid lines correspond to the base case in Section \ref{sec:AortaCC}. Panel (c): 3D view with color bar denoting $R_o$. Panel (d): same as in (c) with $G_o$ values in the color bar.}
\end{figure}

In this numerical example, we consider a situation where the $G_o$ parameter has a negative perturbation in a localized lateral section of the artery near the upstream boundary and $R_o$ is also increased in the same area, as specified below. This situation is associated with an idealized thoracic aortic aneurysm \cite{ho2017modelling}, where the artery's wall is less rigid in a localized zone. The two parameters $G_o$ and $R_o$ are given by

\begin{figure}[h!]
\begin{center}
{\includegraphics[width=\textwidth]{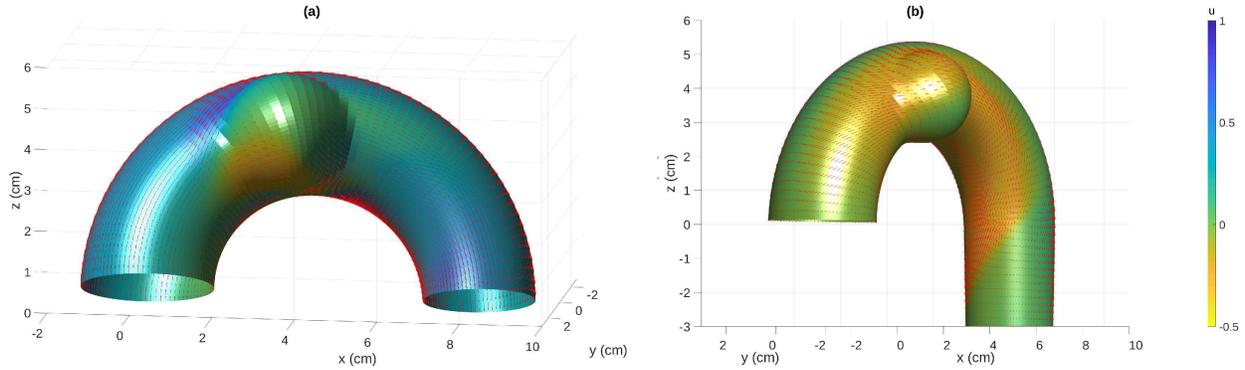}}
\end{center}
\caption{\label{fig:Ex5_Vortex} Circulation pattern in an idealized aorta vessel with a bulge. Three dimensional views of the artery with the velocity field at $t= 0.2 \text{ s}$ (a) and $t=0.4 \text{ s}$ (b) are shown where the parameters $G_o$ and $R_o$ are given by equations \eqref{Ex5:G_o} and \eqref{Ex5:R_o}. The arrows indicate the 3D velocity field given by equation \eqref{eq:3DVField}. }
\end{figure}

\begin{equation}\label{Ex5:G_o}
G_o\left(s,\theta\right) = \left\{
\begin{array}{ccc}
&& \\
G^\star_o\left(s\right) & \text{  if  } & d\left(s,\theta\right) > h\left(\theta\right) R^\star_o\left(s\right), \\
&& \\
\left[1 - \frac{1}{2}\sin\left(\left[1-\frac{d\left(s,\theta\right)}{h\left(\theta\right) R_o^\star\left(s\right)}\right]\frac{\pi}{2}\right)\right] G^\star_o\left(s\right) & \text{  if  } & d\left(s,\theta\right) \leq h\left(\theta\right) R^\star_o\left(s\right), \\
&&
\end{array}
\right.
\end{equation}
and
\begin{equation}
\label{Ex5:R_o}
R_o\left(s,\theta\right) = \left\{
\begin{array}{ccc}
&& \\
h\left(\theta\right) R^\star_o\left(s\right) & \text{if} & d\left(s,\theta\right) > h\left(\theta\right) R^\star_o\left(s\right), \\
&& \\
\left[1 + \frac{3}{4}\sin\left(\left[1-\frac{d\left(s,\theta\right)}{h\left(\theta\right) R^\star_o\left(s\right)}\right]\frac{\pi}{2}\right)\right]h\left(\theta\right) R^\star_o\left(s\right) & \text{if} & d\left(s,\theta\right) \leq h\left(\theta\right) R^\star_o\left(s\right), \\
&&
\end{array}
\right.
\end{equation}
where 
$$
d\left(s,\theta\right) = \sqrt{\left[x\left(s^\star,\theta^\star\right) - x\left(s,\theta\right)\right]^2 + \left[y\left(s^\star,\theta^\star\right) - y\left(s,\theta\right)\right]^2 + \left[z\left(s^\star,\theta^\star\right) - z\left(s,\theta\right)\right]^2},
$$
$$ s^\star = 5 \text{ cm},\theta^\star = \pi \text{ rad}.$$

A 3D view of the artery at times $t=0.2 \text{ }$ and $t=0.4 \text{ s}$ can be found in panels (a) and (b) of Figure \ref{fig:Ex5_Vortex}, respectively. In each panel, the velocity field is shown to analyze the change in the dynamic. The colobar shows the axial velocity contours. Panel (a) shows a section of the artery near the perturbation. The bulge induces a vortex-like structure at time $t= 0.2\text{ s}$ in the lower region of the bulge but the flow moves in the downstream direction after it passes that region where the artery is less rigid. Panel (b) shows the velocity field in a larger area at time $t=0.4 \text{ s}$. The circulation pattern that was initially found near the bulge has now been extended to a larger region, where the axial velocity is negative in the right side of the artery ($\pi/2 \le \theta \le 3\pi/2$) near the bulge and positive on the opposite side. Circulation patterns can be found on other simulations. See for instance \cite{tan2009analysis}.


\section{Conclusions}

Mathematical models for blood flows can serve as a tool for evaluations before a surgical treatment. The simulations rely on accurate numerical solutions to the corresponding Partial Differential Equations. A timely evaluation requires a prompt numerical solution. Three dimensional models provide detailed information of the fluid's evolution, giving accurate and realistic results. However, they involve a high computational cost and are not always a practical tool for the above goals. As an alternative, one dimensional models have been derived in the literature. Those models consist of limiting equations that assume the cross sections to be circular with a small radius when compared to the artery's length. Of course, those models involve a low computational cost but they are limited by the conditions used to derive them. Although they have shown to be useful to simulate pressure waves, one looses detailed information of the artery's evolution. In this work, we have presented a new intermediate two-dimensional model that allows for arbitrary cross sections. The limiting model is valid for small cross-sectional ratios and other reasonable assumptions. We present this model as an alternative with a better balance between realism and computational cost. The resulting system is conditionally hyperbolic and the spectral properties are described. We have also provided a well-balanced positivity-resulting central-upwind scheme to obtain numerical results. We tested it in idealized aorta models with damaged areas among other scenarios. Our results showed that one can obtain simulations with more details in the artery's evolution with a model that has a much lower computational cost when compared to three-dimensional models. 

\section*{Acknowledgments}
C. Rosales-Alcantar was supported by Conacyt National Grant 701892. G. Hernandez-Duenas was supported in part by grants UNAM-DGAPA-PAPIIT IN113019 \& Conacyt A1-S-17634. Some simulations were performed at the Laboratorio Nacional de Visualizaci\'on Cient\'ifica Avanzada at UNAM Campus Juriquilla, and the Authors received technical support from Luis Aguilar, Alejandro De Le\'on, and Jair Garc\'ia from that lab.

\appendix

\section{Derivation of the model}
\label{appendix:Derivation}

The model is derived in this appendix. The first step is the description of the Navier-Stokes equations in cylindrical coordinates. For that end, let us define the gradients in cartesian and cylindrical coordinates by 
\begin{equation}
\grad = \left(\partx,\party,\partz\right),\:\:\: \text{ and } \:\:\: \grad_c = \left(\parts,\partr,\parth\right),
\end{equation}
respectively, where $(x,y,z)$ and $(s,r,\theta)$ are related by equation \eqref{ncs}. The corresponding velocity fields are given by
\begin{equation}
\vect V = \left(V_x,V_y,V_z\right),\:\:\:\:\:\: \vect V_c = \left(V_s,V_r,V_{\theta}\right).
\end{equation}

Applying change of variables, we find that the velocity field in the cylindrical coordinates is given by
\begin{equation}\label{nvv}
\begin{array}{rcl}
V_s & = & \frac{r}{\detj}\Big\{\cosa V_x + \sina V_z\Big\},\\
V_r & = & -\sina\sint V_x + \cost V_y + \cosa\sint V_z, \\
V_{\theta} & = & \frac{1}{r}\Big\{ -\sina\cost V_x - \sint V_y + \cosa\cost V_z \Big\}.
\end{array}
\end{equation}

The partial derivatives in cylindrical coordinates are given in terms of the derivatives in cartesian coordinates by
\begin{equation}
\begin{array}{rcl}
\parts & = & \frac{\detj}{r}\Big\{ \cosa \partx + \sina \partz \Big\} \\
\partr & = & -\sina\sint\partx + \cost\party + \cosa\sint\partz \\
\parth & = & r\Big\{-\sina\cost\partx -\sint\party + \cosa\cost\partz\Big\}.
\end{array}
\end{equation}

We take the incompressible Navier-Stokes equations with varying density as the full system to be reduced. Such system can be written as
\begin{equation} \label{nsg_cartesian}
\begin{array}{lcl}
\dermat \; \rho & = & 0,\\
\dermat\left(\rho V_x\right) & = & -\partx \; P + \nu \lap \; V_x , \\
\dermat\left(\rho V_y\right) & = & -\party \; P + \nu \lap \; V_y , \\
\dermat\left(\rho V_z\right) & = & -\partz \; P + \nu \lap \; V_z  - \rho g, \\
\grad \cdot \vect V & = & 0.
\end{array}
\end{equation}
We need to re-write the divergence, material derivative, and Laplacian in cylindrical coordinates. Let $\left(F_1,F_2,F_3\right)$ a vectorial field. Then,
\begin{equation}
\grad\cdot\left(F_1,F_2,F_3\right) = \frac{1}{\vert J\vert }\grad_c\cdot\left[\vert J\vert \left(\tilde{F}_1,\tilde{F}_2,\tilde{F}_3\right)\right],
\end{equation}

where the vector field $\left(\tilde{F}_1,\tilde{F}_2,\tilde{F}_3\right)$ is given by
\begin{equation}
\begin{array}{rcl}
\tilde{F}_1 & = & \frac{r}{\detj} \Big\{\cosa F_1 + \sina F_3\Big\}, \\
\tilde{F}_2 & = & \cost F_2 + \sint\left[-\sina F_1 + \cosa F_3\right] ,\\
\tilde{F}_3 & = & \frac{1}{r}\Big\{-\sint F_2 + \cost\left[-\sina F_1 + \cosa F_3\right]\Big\}.
\end{array}
\end{equation}

In cylindrical coordinates, the material derivative can be expressed as
\begin{equation}
\begin{array}{rcl}
\dfrac{Df}{Dt} & = & \partt\left(f\right) + \vect V_c \cdot \grad_c\left(f\right) \\
    & = & \dfrac{1}{\detj}\Big\{\partt\left(\detj f\right) + \grad_c\cdot\left(\detj f \vect V_c\right) - f \grad_c \cdot \left(\detj \vect V_c\right) \Big\}.
\end{array}
\end{equation}

In the case of incompressible fluids, the last term vanishes and we obtain
\begin{equation}
\frac{Df}{Dt} = \frac{1}{\detj}\Big\{\partt\left(\detj f \right) + \parts\left(\detj f V_s\right)+\partr\left(\detj f V_r\right) + \parth\left(\detj f V_\theta\right)\Big\}.
\end{equation}

Furthermore, the Laplacian can be expressed as
\[
\lap\left(f\right) = \frac{1}{\detj}\parts\Bigg(\frac{r^2}{\detj}\parts \; f \Bigg) + \frac{1}{\detj}\partr\Big(\detj \partr \;  f \Big) + \frac{1}{\detj}\parth\Bigg(\frac{\detj}{r^2} \parth \; f \Bigg).
\]

Straightforward but long calculations gives the Navier-Stokes equations \eqref{ncs} in cylindrical variables. The new system is given by 
\begin{equation}
\label{nsg_cylindrical}
\begin{array}{lcl}
\dermat\left(\rho\right)  & = & 0,\\
\dermat\left(\rho \left[\detjr\right]^2\vs\right) & = & - \parts P_2 + \detjr\parts\left(\detjr\right)\rho \vs^2 - \sina\rho g, \\
& & + \nu\Bigg\{ \lap\left(\left[\detjr\right]^2\vs\right) - \lap\left(\left[\detjr\right]^2\right)\vs + 2\rdetj\partr\left[\detjr\right]\parts\left(\vr\right) \\ 
& &  + 2\rdetj\parth\left(\detjr\right)\parts\left(\vth\right) + \grad_c\left(\rdetj\parts\left(\detjr\right)\right)\cdot\vect V_c \Bigg\},\\
\dermat\left(\rho V_r\right) & = &  -\partr\left(P_2\right) + \detjr\partr\left(\detjr\right)\rho \vs^2 + r\rho\vth^2 \\
& & + \nu\left\{ \lap\left(\vr\right) -2\rdetj\partr\left(\detjr\right)\parts\left(\vs\right) - \frac{2}{r}\parth\left(\vth\right) \right. \\
&& - \left(\rdetj\right)^2\parts\left(\detjr\partr\left(\detjr\right)\right)\vs \\
& & \left. - \frac{1}{\detj^2}\left[\left(\detjr\right)^2 + \left(\detjr -1\right)^2\right]\vr - \frac{\partr\left(\detj\right)\parth\left(\detj\right)}{\detj^2}\vth \right\}, \\
\dermat\left(\rho r^2 V_\theta\right) & = & - \parth\left(P_2\right) + \detjr\parth\left(\detjr\right)\rho\vs^2 \\
& & + \nu\Bigg\{ \lap\left(r^2\vth\right) - 2\rdetj\parth\left(\detjr\right)\parts\left(\vs\right)  \\
& & + \frac{2}{r}\parth\left(\vr\right) - \frac{2}{\detj}\partr\left(\detj r\vth\right) - \left[\rdetj\right]^2\parts\left(\detjr\parth\left(\detjr\right)\right)\vs \\
&& + \frac{r}{\detj^2}\parth\left(\detjr\right)\vr - \left[\rdetj\parth\left(\detjr\right)\right]^2 \vth \Bigg\}, \\
 \grad_c \cdot \left(\detj\vect V_c\right) & = & 0,
\end{array}
\end{equation}
where $P_2 = P+r \cos(\alpha(s)) \sin(\theta) \rho g$ is the transmural pressure.

\subsection{The reduced equations}

We carry out an asymptotic analysis to remove small terms in the equations that do not add a significant contribution in the budget and allows us to simplify the model. Following \cite{vcanic2003mathematical}, we define $V_{s,0}$, $V_{r,0}$, and $V_{\theta,0}$ be the characteristic radial, axial and angular velocities. Let also $\lambda$ and $R_0$ be the characteristic axial and radial lengthscales. Each quantity is non-dimensionalized as $r=R_0 \tilde{r}$, $s = \lambda \tilde{s}$, $t = \frac{\lambda}{V_0} \tilde{t}$, $V_s = V_{s,0} \tilde{V}_s$, $V_r = V_{r,0}\tilde{V}_r$, $V_\theta = V_{\theta,0} \tilde{V}_\theta$, $P=\rho V_0^2 \tilde{P}$. Following equation \eqref{eq:Epsilon}, the small parameter in this expansion is the ratio between radial and axial lengthscales
\[
\epsilon : = \frac{R_0}{\lambda}=\frac{V_{r,0}}{V_{s,0}}.
\]
This is a reasonable assumption because this ratio is about $\frac{R_0}{\lambda} = \mathcal{O}\left(10^{-2}\right)$ for the aorta between the renal and iliac arteries 
. \\

The non-dimensional version of the model is given by
\begin{equation}
\begin{array}{lcl}
\tdermat\left(\trho\right) & = & 0 \\
\tdermat\left(\trho \left[\tdetjr\right]^2 \tVs\right) & = & -\blue{\frac{\left[P\right]}{\rho_0 V_{s,0}^2}}\tparts \; \tiP_2  - \blue{\frac{gT}{V_{s,0}}}\sina \trho  + \tdetjr\tparts\left(\tdetjr\right)\trho \tVs^2  \\
& & + \blue{\frac{\nu T}{\rho_0 R_0^2}}\Bigg\{\blue{\left(\frac{R_0}{\lambda}\right)^2} \frac{1}{\tdetj}\tparts\left(\frac{\tir^2}{\tdetj}\tparts\left(\left[\tdetjr\right]^2\tVs\right)\right) \\
&& + \frac{1}{\tdetj}\tpartr\left(\tdetj\tpartr\left(\left[\tdetjr\right]^2\tVs\right)\right) + \frac{1}{\tdetj}\tparth\left(\frac{\tdetj}{\tir^2}\tparth\left(\left[\tdetjr\right]^2\tVs\right)\right)  \\
& & + 2\frac{\tir}{\tdetj}\tpartr\left(\frac{\tdetj}{\tir}\right)\tparts\left(\tVr\right) + 2\frac{\tir}{\tdetj}\tparth\left(\frac{\tdetj}{\tir}\right)\tparts\left(\tVt\right) \\
& & +\blue{\left(\frac{R_0}{\lambda}\right)^2} \tgradc\left(\frac{\tir}{\tdetj}\tparts\left(\frac{\tdetj}{\tir}\right)\right)\cdot \tvector \\
& &  -   \blue{\left(\frac{R_0}{\lambda}\right)^2} \frac{1}{\tdetj}\tparts\left(\frac{\tir^2}{\tdetj}\tparts\left(\left[\tdetjr\right]^2\right)\right)\tVs \\
&& - \frac{1}{\tdetj}\tpartr\left(\tdetj\tpartr\left(\left[\tdetjr\right]^2\right)\right)\tVs - \frac{1}{\tdetj}\tparth\left(\frac{\tdetj}{\tir^2}\tparth\left(\left[\tdetjr\right]^2\right)\right)\tVs \Bigg\}\\
\tdermat\left(\trho \tVr\right) & = & -\blue{\frac{\left[P\right]}{\rho_0 V_{r,0}^2}}\tpartr \; \tiP_2+ \blue{\left(\frac{V_{s,0}}{V_{r,0}}\right)^2} \frac{\tdetj}{\tir}\tpartr\left(\frac{\tdetj}{\tir}\right)\trho \tVs^2  + \trho \tir \tVt^2 \\
& & + \blue{\frac{\nu T}{\rho_0 R_0^2}}\Bigg\{
\blue{\left(\frac{R_0}{\lambda}\right)^2} \frac{1}{\tdetj}\tparts\left(\frac{\tir^2}{\tdetj}\tparts\left(\tVr\right)\right) + \frac{1}{\tdetj}\tpartr\left(\tdetj\tpartr\left(\tVr\right)\right) \\
& &  + \frac{1}{\tdetj}\tparth\left(\frac{\tdetj}{\tir^2}\tparth\left(\tVr\right)\right) - 2\frac{\tir}{\tdetj}\tpartr\left(\frac{\tdetj}{\tir}\right)\tparts\left(\tVs\right) - \frac{2}{\tir}\tparth\left(\tVt\right) \\
&&  - \left(\frac{\tir}{\tdetj}\right)^2\tpartr\left(\frac{\tdetj}{\tir}\right) \tgradc\left(\frac{\tdetj}{\tir}\right)\cdot \tvector \\
& & - \frac{\tir}{\tdetj}\tparts\tpartr\left(\frac{\tdetj}{\tir}\right) \tVs - \frac{1}{\tir^2}\tVr - \frac{1}{\tdetj} \tparth\left(\frac{\tdetj}{\tir}\right)\tVt\Bigg\}  \\
\tdermat\left(\trho \left(\tir\right)^2\tVt\right) & = & -\blue{\frac{\left[P\right]}{\rho_0 R_0^2 V_{\theta,0}^2}}\tparth \; \tiP_2 + \blue{\left(\frac{V_{s,0}}{R_0 V_{\theta,0}}\right)^2}\frac{\tdetj}{\tir}\tparth\left(\frac{\tdetj}{\tir}\right)\trho \tVs^2 \\
& & + \blue{\frac{\nu T}{\rho_0 R_0^2}}\Bigg\{
\blue{\left(\frac{R_0}{\lambda}\right)^2} \frac{1}{\tdetj}\tparts\left(\frac{\tir^2}{\tdetj}\tparts\left(\left(\tir\right)^2\tVt\right)\right)  + \frac{1}{\tdetj}\tpartr\left(\tdetj\tpartr\left(\left(\tir\right)^2\tVt\right)\right) \\
& & + \frac{1}{\tdetj}\tparth\left(\frac{\tdetj}{\tir^2}\tparth\left(\left(\tir\right)^2\tVt\right)\right) - \frac{2\tir}{\tdetj}\tparth\left(\frac{\tdetj}{\tir}\right)\tparts\left(\tVs\right) + \frac{2}{\tir}\tparth\left(\tVr\right) \\
&& - \frac{2}{\tdetj}\tpartr\left(\tdetj\tir\tVt\right) - \left[\frac{\tir}{\tdetj}\right]^2\tparts\left(\frac{\tdetj}{\tir}\tparth\left(\frac{\tdetj}{\tir}\right)\right)\tVs  \\
& & - \frac{\tir}{\tdetj^2}\tparth\left(\frac{\tdetj}{\tir}\right)\tVr - \left[\frac{\tir}{\tdetj}\tparth\left(\frac{\tdetj}{\tir}\right)\right]^2 \tVt \Bigg\}  \\
\tgradc \cdot \left(\tdetj\tvector\right) & = & 0. 
\end{array}
\end{equation}

We recall the assumptions \eqref{eq:Assumption1} and \eqref{eq:Assumption2}, given by
\[
\frac{[P]}{\rho_o V_{s,o}^2} = O(1), \; \frac{V_{s,o}}{R_o V_{\theta,o}} = O(1), \text{ and } \frac{gT}{V_{s,o}}= O(1),
R_o \left| \alpha'(s) \right| = O\left( \epsilon \right),
\]
and 
\[
\frac{\nu T}{\rho_o R_o^2} = O\left( \epsilon \right).
\]

There is just one leading order term in the momentum equation in the radial direction that is found as follows. The first term in the right-hand side has a factor of
\[
\frac{[P]}{\rho_o V_{r,o}^2} =  \frac{[P]}{\rho_o V_{s,o}^2}  \frac{V_{s,o}^2}{V_{r,o}^2} = O(\epsilon^{-2}). 
\]
The second term has a factor of order $O(\epsilon^{-2})$. However, we ignore that term because $$\partial_{r} (|J|/r) = -\sin(\theta) \alpha'(s) R_o = O(\epsilon).$$ The viscosity terms are all order $O( \epsilon )$ or higher, and the left-hand side is $O(1)$. Thus, taking the leading order term, we obtain
\[
\partial_{\tilde r}\left(\tilde P_2\right)  = 0,
\]
which implies that $\tilde P_2$ is independent of $\tilde r$. \\

In the equation of balance for the angular momentum, we will exclude only terms that are order $O(\epsilon)$ or higher to keep the contribution of the artery's curvature on the flow. The first term in the right-hand side has a factor
\[
\frac{[P]}{\rho_o (R_o V_{\theta,o})^2 } = \frac{[P]}{\rho_o V_{s,o}^2} \frac{V_{s,o}^2}{(R_o V_{\theta,o})^2} = O(1), 
\]
and we keep it. As discussed above, the non-dimensional parameter involving the viscosity term is order $O( \epsilon )$. For the terms inside the brackets, we assume that $rV_\theta, V_r$ and $|J|/r$ depend all weakly on $\theta$, which is consistent with the fact that the blood flow moves mainly in the axial direction. As a result, only two terms in front of the viscosity coefficient has a leading contribution, as specified below in equation \eqref{eq:ReducedEqns}.\\

Similarly, we only exclude terms in the momentum equation that are order $O(\epsilon)$ or higher. All the terms before the viscosity coefficient are order $O(1)$ or $O(\epsilon)$. Only one viscosity term inside the brackets has a leading contribution. The other terms have either a factor of order $O(\epsilon)$, or can be neglected due to the weak dependance on $\theta$.\\

The reduced system in dimensional form is given in equation \eqref{eq:ReducedEqns} in Section \ref{sec:AveragedEqns} . In that equation we redefine $p=P_2$ as the transmural pressure to avoid heavy notation. In Section \ref{sec:AveragedEqns} , those equations are then integrated in the radial direction to derive the desired model \eqref{eq:MainSystemNotConservation}.

\clearpage
\bibliographystyle{unsrt}
\bibliography{refs}

\end{document}